\documentclass[12pt]{amsart}

\usepackage{amsmath}
\usepackage{amssymb}
\usepackage{pdfsync}

\newcommand{\C}{{\mathbb C}}       
\newcommand{\R}{{\mathbb R}}       
\newcommand{\Z}{{\mathbb Z}}       
\newcommand{\DD}{{\mathcal D}}

\newcommand{\HH}{{\mathcal H}}

\newcommand{\WW}{{\mathcal W}}

\newcommand{\CC}{{\mathcal C}}

\newcommand{\diam}{{\rm diam}}
\newcommand{\dist}{{\rm dist}}

\newcommand{\imag}{{\rm Im}}

\newcommand{\rf}[1]{{(\ref{#1})}}

\newcommand{\vphi}{{\varphi}}
\newcommand{\ve}{{\varepsilon}}
\newcommand{\vv}{{\vspace{2mm}}}
\newcommand{\vvv}{{\vspace{3mm}}}
\newcommand{\wt}[1]{{\widetilde{#1}}}
\newcommand{\wh}[1]{{\widehat{#1}}}

\newcommand{\rest}{{\lfloor}}

\newtheorem{theorem}{Theorem}[section]
\newtheorem{lemma}[theorem]{Lemma}

\newtheorem{coro}[theorem]{Corollary}

\newtheorem*{theorem*}{Theorem}

\theoremstyle{definition}

\theoremstyle{remark}
\newtheorem{rem}[theorem]{Remark}

\numberwithin{equation}{section}

\newcommand{\brem}{\begin{rem}}
\newcommand{\erem}{\end{rem}}

\begin{document}

\title[Regularity of domains in terms of the
Beurling transform]{Regularity of $\CC^1$ and Lipschitz domains in terms of the
Beurling transform}

\author{Xavier Tolsa}

\address{Xavier Tolsa. Instituci\'{o} Catalana de Recerca
i Estudis Avan\c{c}ats (ICREA) and Departament de
Ma\-te\-m\`a\-ti\-ques, Universitat Aut\`onoma de Bar\-ce\-lo\-na,
Catalonia} \email{xtolsa@mat.uab.cat}

\thanks{Partially supported by grants
2009SGR-000420 (Generalitat de
Catalunya) and MTM-2010-16232 (Spain).}

\maketitle

\begin{abstract}
Let $\Omega\subset\C$ be a bounded $\CC^1$ domain, or a Lipschitz 
domain ``flat 
enough'', and consider the Beurling transform of $\chi_\Omega$:
$$B\chi_\Omega(z)=\lim_{\ve\to0}\frac{-1}\pi\int_{w\in\Omega,|z-w|>\ve}\frac1{(z-w)^2}\,dm(w).$$
Using a priori estimates, in this paper we solve the following free boundary problem: if $B\chi_\Omega$
belongs to the Sobolev space $W^{\alpha,p}(\Omega)$ for  $0<\alpha\leq1$, $1<p<\infty$ such that 
$\alpha p>1$, then the outward unit normal $N$ on $\partial\Omega$ 
is in the Besov space ${B}_{p,p}^{\alpha-1/p}(\partial\Omega)$. 
The converse statement, proved previously by Cruz and Tolsa, also holds. So we have
$$B(\chi_\Omega)\in W^{\alpha,p}(\Omega)\quad\Longleftrightarrow \quad N\in {B}_{p,p}^{\alpha-1/p}(\partial\Omega).$$
Together with recent results by Cruz, Mateu and Orobitg, from the preceding equivalence
one infers that the Beurling transform is bounded
in $W^{\alpha,p}(\Omega)$ if and only if 
the outward unit normal $N$  
belongs to ${B}_{p,p}^{\alpha-1/p}(\partial\Omega)$, assuming that $\alpha p>2$.
\end{abstract}


\section{Introduction}

In this paper we show that the boundedness of the Beurling transform in the Sobolev spaces
$W^{\alpha,p}(\Omega)$, with $0<\alpha\leq1$ and $1<p<\infty$ such that $\alpha p> 1$,  characterizes the Besov smoothness of the boundary $\partial\Omega$, whenever $\Omega$ is a $\CC^1$ domain, or a Lipschitz domain ``flat enough''. This can be considered as a free boundary problem.

The Beurling transform of a function $f:\C\to\C$, with $f\in L^p$ for some $1\leq p<\infty$, is
defined by
$$Bf(z)=\lim_{\ve\to0}\frac{-1}\pi\int_{|z-w|>\ve}\frac{f(w)}{(z-w)^2}\,dm(w).$$
It is known that this limit exist a.e. The Beurling transform plays an essential role
in the theory of quasiconformal mappings in the plane, because it intertwines the $\partial$ and
$\bar \partial$ derivatives. Indeed, in the sense of distributions, one has
$$B(\bar\partial f)=\partial f.$$

Let $\Omega\subset\C$ be a bounded domain (open and connected). We say that $\Omega\subset\C$ is a $(\delta,R)$-Lipschitz domain if for each $z\in\partial\Omega$
there exists a Lipschitz function $A:\R\to\R$ with slope $\|A'\|_\infty\leq\delta$ such that, after
a suitable rotation,
$$\Omega\cap B(z,R) = \bigl\{(x,y)\in B(z,R):\,y>A(x)\bigr\}.$$
If we do not care about the constants $\delta$ and $R$, then we just say that $\Omega$ is a
Lipschitz domain. If in this definition we assume the function $A$ to be of class $\CC^1$,
then we say that $\Omega$ is a $\CC^1$ domain.

In \cite{Cruz-Tolsa} it has been shown that for any Lipschitz domain $\Omega$ and 
$0<\alpha\leq1$ and $1<p<\infty$ such that $\alpha p> 1$,
if the outward unit normal is in the Besov space $B_{p,p}^{\alpha-1/p}(\partial\Omega)$, then $B(\chi_\Omega)$ belongs
to the Sobolev space $W^{\alpha,p}(\Omega)$. More precisely, the following estimate has been proved:
\begin{equation}\label{eqct3}
\|B(\chi_\Omega)\|_{\dot W^{\alpha,p}(\Omega)}\leq c\,\|N\|_{\dot{B}_{p,p}^{\alpha-1/p}(\partial\Omega)},
\end{equation}
where $N$ stands for the 
outward normal unitary vector on $\partial \Omega$, 
$\dot W^{\alpha,p}(\Omega)$ is a homogeneous sobolev space on $\Omega$, and
$\dot{B}_{p,p}^{\alpha-1/p}(\partial\Omega)$ is a homogeneous
Besov space on $\partial\Omega$. See the next section for the precise definition of Sobolev and Besov spaces, 
as well as their homogeneous versions. The
constant $c$ in \rf{eqct3} may depend on $p$ and on 
the Lipschitz character of $\Omega$, i.e.\ on $\delta$ and on
 $\HH^1(\partial \Omega)/R$ (here $\HH^1$ stands for the length or $1$-dimensional Hausdorff measure). 
 Observe that, by the $L^p$ boundedness of the Beurling transform,
$$\|B(\chi_\Omega)\|_{W^{\alpha,p}(\Omega)}\leq c\,\bigl(m(\Omega)^{1/p} + \|B(\chi_\Omega)\|_{
\dot W^{\alpha,p}(\Omega)}\bigr).$$
Thus \rf{eqct3} guaranties that $B(\chi_\Omega)\in W^{\alpha,p}(\Omega)$ whenever $N\in \dot{B}_{p,p}^{\alpha-1/p}(\partial\Omega).$ 


Our main result is a (partial) converse to \rf{eqct3}:

\begin{theorem}\label{teodom}
Let $\Omega\subset\C$ be a $(\delta,R)$-Lipschitz domain and 
$0<\alpha\leq1$ and $1<p<\infty$ such that $\alpha p> 1$. If $\delta=\delta(p)>0$ is small enough, then
\begin{equation}\label{eqdf89}
\|N\|_{\dot B_{p,p}^{\alpha-1/p}(\partial\Omega)}\leq c\,\|B(\chi_\Omega)\|_{\dot W^{\alpha,p}(\Omega)}
+ c\,\HH^1(\partial \Omega)^{-\alpha+2/p}.
,
\end{equation}
where $c$ depends on $\delta$ and $p$.
\end{theorem}

Some remarks are in order. 
Notice first that $\CC^1$ domains are $(\delta,R)$-Lipschitz domains for every $\delta>0$ and
an appropriate $R=R(\delta)$. So the theorem applies to all $\CC^1$ domains.
Then, by combining the results from \cite{Cruz-Tolsa} with Theorem \ref{teodom}, one infers that,
for a $\CC^1$ domain $\Omega$ and $\alpha,p$ as above, 
\begin{equation}\label{eqdob22}
B(\chi_\Omega)\in W^{\alpha,p}(\Omega)\quad\Longleftrightarrow \quad N\in {B}_{p,p}^{\alpha-1/p}(\partial\Omega).
\end{equation}


Let us remark that the inequality
$$\|N\|_{\dot{B}_{p,p}^{\alpha-1/p}(\partial\Omega)}\leq c\,\|B(\chi_\Omega)\|_{\dot W^{\alpha,p}(\Omega)}$$
fails in general. Indeed, when $\Omega$ is an open ball it turns out that $B(\chi_\Omega)$ vanishes identically on $\Omega$. So $\|B(\chi_\Omega)\|_{\dot W^{\alpha,p}(\Omega)}=0$, while $\|N\|_{\dot{B}_{p,p}^{\alpha-1/p}(\partial\Omega)}\neq0$, because $N$ is not constant.

Recall that the Besov spaces $B_{p,p}^{\alpha-1/p}$ appear naturally in the context of Sobolev spaces.
Indeed, a well known theorem of Gagliardo \cite{Gagliardo} asserts that the traces of the functions from $W^{1,p}(\Omega)$ on $\partial\Omega$ coincide with the functions from 
$B_{p,p}^{1-1/p}(\partial\Omega)$, whenever $\Omega$ is a Lipschitz domain. An analogous result holds for 
$0<\alpha<1$.
So, by this result theorem with \rf{eqdob22}, one deduces that $B(\chi_\Omega)\in W^{\alpha,p}(\Omega)$
if and only if $N$ is the trace of some (vectorial) function from $W^{\alpha,p}(\Omega)$, which
looks a rather surprising statement at first sight.

Our motivation for the characterization of those domains such that $B\chi_\Omega\in W^{\alpha,p}
(\Omega)$ arises from the results of Cruz, Mateu and Orobitg in \cite{CMO}. 
In this paper the authors study the smoothness of quasiconformal mappings when the Beltrami coefficient belongs to $W^{\alpha,p}(\Omega)$, for some fixed $1<p<\infty$ and $0<\alpha<1$.
As an important step in their arguments, they prove following kind of $T1$ theorem:

\begin{theorem*}[\cite{CMO}]
Let $\Omega\subset\C$ be a $\CC^{1+\ve}$ domain, for some $\ve>0$, and let 
$0<\alpha\leq1$ and $1<p<\infty$ be such that $\alpha p> 2$.
Then, the Beurling transform is bounded
in $W^{\alpha,p}(\Omega)$ if and only if $B(\chi_\Omega)\in W^{\alpha,p}(\Omega)$.
\end{theorem*}

As a corollary of the preceding result and Theorem \ref{teodom} we obtain the following.

\begin{coro}
Let $\Omega\subset\C$ be a Lipschitz domain 
and let 
$0<\alpha\leq1$ and $1<p<\infty$ be such that $\alpha p> 2$. Then, the Beurling transform is bounded
in $W^{\alpha,p}(\Omega)$ if and only if the outward unit normal of $\Omega$ is in the Besov space $ B_{p,p}^{\alpha-1/p}(\partial\Omega)$.
\end{coro}

Observe that the fact that $N\in \dot B_{p,p}^{\alpha-1/p}(\partial\Omega)$ implies that the local parameterizations of the boundary can be taken from $B_{p,p}^{1+\alpha-1/p}(\R
)\subset \CC^{1+\ve}(\R)$ because $\alpha p>2$, and thus the theorem from Cruz-Mateu-Orobitg applies.
For more details, see Lemma \ref{lemanorm} below.

On the other hand, it is worth mentioning  that the boundedness of the Beurling transform in the Lipschitz spaces $\rm Lip_\ve(\Omega)$ for
domains $\Omega$ of class $\CC^{1+\ve}$ has been studied in \cite{Mateu-Orobitg-Verdera}, \cite{Li-Vogelius},
and \cite{Depauw}, because of the applications to quasiconformal mappings and PDE's.

We will prove Theorem \ref{teodom} by reducing it to the case of the so called special Lipschitz
domain, where $\Omega$ is the open set lying above a Lipschitz graph. That is, given a Lipschitz 
function $A:\R\to\R$, one sets 
\begin{equation}\label{eqspec4}
\Omega=\{(x,y)\in\C:\,y>A(x)\}.
\end{equation}
In this situation, we will show the following:

\begin{theorem}\label{teopri}
Let $A:\R\to\R$ be a Lipschitz function with compact support and consider the special Lipschitz domain $\Omega$ defined in \rf{eqspec4}.
For 
$0<\alpha\leq1$ and $1<p<\infty$ be such that $\alpha p> 1$,
there exists $\delta=\delta(\alpha,p)>0$ small enough such if $\|A'\|_\infty\leq\delta$, then
\begin{equation}\label{eqde47}
\|A\|_{\dot{B}_{p,p}^{1+\alpha-1/p}}\leq c\,\|B(\chi_\Omega)\|_{\dot W^{\alpha,p}(\Omega)},
\end{equation}
with $c$ depending on $\alpha,p$.
\end{theorem}

Above, $\dot{B}_{p,p}^{1+\alpha-1/p}$ and $\dot W^{\alpha,p}(\Omega)$ stand for homogeneous Besov  
and Sobolev  spaces, on $\R$ and on $\Omega$, respectively.

Let us remark that the converse inequality 
\begin{equation}\label{eqct6}
\|B(\chi_\Omega)\|_{\dot W^{\alpha,p}(\Omega)}\leq c(\delta)\,\|A\|_{\dot{B}_{p,p}^{1+\alpha-1/p}}
\end{equation}
also holds for special Lipschitz domains, without the smallness assumption on $\delta$. This has been
shown in \cite{Cruz-Tolsa}. Notice, in particular, that \rf{eqct6} shows that $B\chi_\Omega$ is constant in $\Omega$ if this is a half plane. Of course, this can be proved without appealing to \rf{eqct6}, in a much more elementary way. This property plays a key role in the arguments in \cite{Cruz-Tolsa} and also in the ones of the present paper.

It is easy to check that
\begin{equation}\label{eqnn}
\|A\|_{\dot{B}_{p,p}^{1+\alpha-1/p}}\approx \|N\|_{\dot{B}_{p,p}^{\alpha-1/p}(\partial\Omega)},
\end{equation}
where, as above, $N(z)$ stands for the outward unitary normal at $z$.  So \rf{eqde47} is analogous to \rf{eqdf89}. For the detailed arguments,
 see Lemma \ref{lemanorm} below. 
 
We will prove Theorem \ref{teopri} by means of a Fourier type estimate. To this end, we will need to approximate the Lipschitz graph by lines at many different scales.
We will estimate the errors in the approximation in terms of the so called $\beta_1$ coefficients.
Given an interval $I\subset\R$ and a function $f\in L^1_{loc}$, one sets
\begin{equation}\label{defbeta} 
\beta_1(f,I) = \inf_\rho \frac1{\ell(I)}\int_{3I} \frac{|f(x)-\rho(x)|}{\ell(I)}\,dx,
\end{equation}
where the infimum is taken over all the affine functions $\rho:\R\to\R$.
The coefficients $\beta_1$'s (and other variants $\beta_p$, 
$\beta_\infty$,\ldots) appeared first in the works of Jones \cite{Jones-traveling} and David and Semmes
\cite{DS1} on quantitative rectifiability. 
They have become a useful tool in problems which involve geometric measure theory and multi-scale analysis. See \cite{DS2}, \cite{Leger}, \cite{Mas-Tolsa}, \cite{Tolsa-bilip}, or \cite{Tolsa-jfa}, for example, besides the aforementioned references. In the present paper
we will use the $\beta_1$'s to measure the Besov smoothness of the boundary
of Lipschitz domains, by means of a characterization of Besov spaces in terms of $\beta_1$'s due to
Dorronsoro \cite{Dorronsoro2}.

The plan of the paper is the following. In Section \ref{sec2}, some preliminary notation and background is 
introduced. In particular, several characterizations of Besov spaces are described.
In Section \ref{sec3} we prove some auxiliary lemmas which will be used later.
Theorem \ref{teopri} is proved in Section \ref{sec4}, and then this is used in Section \ref{sec5}
to deduce Theorem \ref{teodom}. In the final Section \ref{sec6} we show that Theorems \ref{teodom}
and \ref{teopri} also hold replacing the $\dot W^{\alpha,p}$ seminorm of $B\chi_\Omega$ by the
$\dot B_{p,p}^\alpha(\Omega)$ one, for $0<\alpha<1$.
\vv


\vvv
\section{Preliminaries}\label{sec2}

\subsection{Basic notation}

As usual, in the paper the letter `$c$' stands for an absolute
constant which may change its value at different occurrences. On
the other hand, constants with subscripts, such as $c_0$, retain
their values at different occurrences. The notation $A\lesssim B$
means that there is a fixed positive constant $c$ such that
$A\leq cB$. So $A\approx B$ is equivalent to $A\lesssim B \lesssim
A$. 

The notation $I(x,r)$ stands for an interval in $\R$ with center $x$ and radius $r$.


\subsection{Dyadic and Whitney cubes}\label{subsec2.1}

By a cube in $\R^n$ we mean a cube with edges parallel to the axes. Most of the cubes in our paper will be
dyadic cubes, which are assumed to be half open-closed. The collection of all dyadic cubes is denoted by $\DD(\R^n)$. They are called intervals for $n=1$ and squares for $n=2$.
The side length of a cube $Q$ is written as $\ell(Q)$, and its center as $z_Q$. 
The lattice of dyadic cubes of side length $2^{-j}$ is
denoted by $\DD_j(\R^n)$. 
Also, given $a>0$ and any cube $Q$, we denote by $a\,Q$ the cube concentric with $Q$ with side length $a\,
\ell(Q)$.

Recall that any open subset $\Omega\subset\R^n$ can be decomposed in the so called Whitney cubes, as follows:
$$\Omega = \bigcup_{k=1}^{\infty} Q_k, $$
where $Q_k$ are disjoint dyadic cubes (the ``Whitney cubes'') such
that for some constants $r>20$ and $D_0\geq1$ the following holds,
\begin{itemize}
\item[(i)] $5Q_k \subset \Omega$.
\item[(ii)] $r Q_k \cap \Omega^{c} \neq \varnothing$.
\item[(iii)] For each cube $Q_k$, there are at most $D_0$ squares $Q_j$
such that $5Q_k \cap 5Q_j \neq \varnothing$. Moreover, for such squares $Q_k$, $Q_j$, we have 
$\frac12\ell(Q_k)\leq
\ell(Q_j)\leq 2\,\ell(Q_k)$.
\end{itemize}
We will denote by $\WW(\Omega)$ the family $\{Q_k\}_k$ of Whitney cubes of $\Omega$.

If $\Omega\subset\C$ is a Lipschitz domain, then $\partial\Omega$ is a chord arc curve.
Recall that a chord arc curve is just the bilipschitz image of 
a circumference. Then one can define a family $\DD(\partial\Omega)$ of ``dyadic'' arcs  which play the same role
as the dyadic intervals in $\R$: for each $j\in\Z$ such that $2^{-j}\leq \HH^1(\partial\Omega)$, $\DD_j(\partial\Omega)$ is a partition of
$\partial \Omega$ into pairwise disjoint arcs of length $\approx2^{-j}$, and 
$\DD(\partial\Omega)=\bigcup_j \DD_j(\partial\Omega)$. As in the case of $\DD(\R^n)$, two
arcs from $\DD(\partial\Omega)$ either are disjoint or one contains the other. The construction
of $\DD(\partial\Omega)$ easy: take an arc length parameterization $S^1(0,r_0)\to\partial\Omega$,  consider a dyadic family of arcs of $S^1(0,r_0)$, and
then let $\DD(\partial\Omega)$ be the image of the dyadic arcs from $S^1(0,r_0)$.

Sometimes, the arcs from $\DD(\partial \Omega)$ will be called dyadic ``cubes'', because of the analogy with $\DD(\R^n)$.

If $\Omega$ is a special Lipschitz domain, that is, 
$\Omega=\{(x,y)\in\C:\,y>A(x)\}$, where $A:\R\to\R$ is a Lipschitz function, there exists
an analogous family $\DD(\partial
\Omega)$. In this case, setting $T(x)=(x,A(x))$, one can take $\DD(\partial\Omega)=T(\DD(\R))$, 
for instance.

If $\Omega$ is either a Lipschitz or a special Lipschitz domain, to each
$Q\in\WW(\Omega)$ we assign a cube $\phi(Q)\in\DD(\partial \Omega)$ such that $\phi(Q)\cap\rho Q\neq
\varnothing$ and $\diam(\phi(Q))\approx\ell(Q)$.
So there exists  some big constant $M$ depending on the parameters of the Whitney decomposition and on the chord arc constant of $\partial \Omega$ such that
$$\phi(Q)\subset M\,Q,\qquad \text{and}\qquad Q\subset B(z,M\ell(\phi(Q)))\quad
\mbox{for all $z\in\phi(Q)$}.$$
From this fact, it easily follows that there exists some constant $c_2$ such that for every $Q\in
\WW(\Omega)$,
$$\#\{P\in\DD(\partial\Omega):\,P=\phi(Q)\}\leq c_2.$$


\subsection{Sobolev spaces} \label{subsec2.15}

Recall that for an open domain $\Omega\subset \R^n$, $1\leq p<\infty$, and a positive integer $m$, the Sobolev space $W^{m,p}(\Omega)$
consists of the functions $f\in L^1_{loc}(\Omega)$ such that
$$\|f\|_{W^{m,p}(\Omega)} = \biggl(\,\sum_{0\leq|\alpha|\leq m} \|D^\alpha f\|_{L^p(\Omega)}^p\biggr)^{1/p}
<\infty,$$
where $D^\alpha f$ is the $\alpha$-th derivative of $f$, in the sense of distributions. 
The homogeneous Sobolev seminorm $\dot W^{m,p}$ is defined by
$$\|f\|_{\dot W^{m,p}(\Omega)} :=   \biggl(\,\sum_{|\alpha|=m} \|D^\alpha f\|_{L^p(\Omega)}^p\biggr)^{1/p}.$$

For a non
integer $0<\alpha<1$, one sets
\begin{equation}\label{eqdalfa}
D^{\alpha}f(x)=\left(\int_{\Omega}\frac{|f(x)-f(y)|^{2}}{|x-y|^{n+2\alpha}}\,dm(y)\right)^{\frac{1}{2}},
\end{equation}
and then
$$\|f\|_{W^{\alpha,p}(\Omega)} = \biggl(\|f\|_{L^p(\Omega)}^p + \|D^\alpha f\|_{L^p(\Omega)}^p \biggr)^{1/p}.$$
The homogeneous Sobolev seminorm $\dot W^{\alpha,p}(\Omega)$ equals
$$\|f\|_{\dot W^{\alpha,p}(\Omega)} =  \|D^\alpha f\|_{L^p(\Omega)}.$$


\subsection{Besov spaces} \label{subsec2.2}
In this section we review some basic results concerning Besov spaces. We only consider the $1$-dimensional case, and pay special attention to the homogeneous Besov spaces $\dot B_{p,p}^\alpha$, with $0<\alpha<1$.

Consider a radial $\CC^\infty$ function $\eta:\R\to\R$ whose Fourier transform $\wh \eta$ is supported in the annulus
$A(0,1/2,3/2)$, such that setting $\eta_{k}(x)=\eta_{2^{-k}}(x) = 2^k\,\eta(2^k\,x)$,
\begin{equation}\label{eqnu3}
\sum_{k\in\Z} \wh{\eta_{k}}(\xi)=1\qquad \mbox{for all $\xi\neq0$.}
\end{equation}
Then, for $f\in L^1_{loc}(\R)$, $1\leq p,q<\infty$, and $\alpha>0$, one defines the seminorm 
$$\|f\|_{\dot{B}_{p,q}^{\alpha}}= \left( \sum_{k\in\Z}\|2^{k\alpha}\eta_{k}*f\|_p^q\right)^{1/q},$$
and the norm 
$$\|f\|_{{B}_{p,q}^{\alpha}}=\|f\|_p +\|f\|_{\dot{B}_{p,p}^{2-1/p}}.$$
The homogeneous Besov space $\dot{B}_{p,q}^{\alpha}\equiv \dot{B}_{p,q}^{\alpha}(\R)$ consists of the functions such that $\|f\|_{\dot{B}_{p,q}^{\alpha}}<\infty$, while the functions in the Besov space ${B}_{p,q}^{\alpha}\equiv {B}_{p,q}^{\alpha}(\R)$ 
are those such that $\|f\|_{{B}_{p,q}^{\alpha}}<\infty$.
If one chooses a function different from $\eta$ which satisfies the same properties as $\eta$
above, then one obtains an equivalent seminorm and norm, respectively. Let us remark that the seminorm
 $\|\cdot\|_{\dot{B}_{p,q}^{\alpha}}$ is a norm if one considers functions modulo polynomials.

For any function $f$ and any $h>0$, denote $\Delta_h(f)(x)=f(x+h)-f(x)$.
For $1\leq p,q<\infty$ and $0<\alpha<1$, it turns out that
\begin{equation}\label{eqgg0}
\|f\|_{\dot{B}_{p,q}^\alpha}^p \approx
\int_0^\infty \frac{\|\Delta_h(f)\|_q^p}{h^{\alpha p + 1}}\,dh,
\end{equation}
assuming $f$ to be compactly supported, say. Otherwise the comparability is true modulo polynomials, that
is, above we should replace $\|\Delta_h(f)\|_q$ by 
$$\inf_{p \text{ polynomial}} \|\Delta_h(f+p)\|_q.$$
See \cite[p.\ 242]{Trieble}, for instance. Analogous characterizations hold for Besov spaces with regularity $\alpha\geq1$. 
In this case it is necessary to use differences of higher order.

Observe that, for $p=q$ and $0<\alpha<1$, one has
\begin{equation}\label{eqbpp1}
\|f\|_{\dot{B}_{p,p}^\alpha}^p \approx
\iint_0^\infty \frac{|\Delta_h(f)|^p}{h^{\alpha p + 1}}\,dh\,dx \approx
\iint \frac{|f(x)-f(y)|^p}{|x-y|^{\alpha p + 1}}\,dx\,dy,
\end{equation}
for $f$ with compact support.
These results motivate the definition of the $\dot{B}_{p,p}^\alpha$-seminorm 
over a chord arc curve. If $\Gamma$ is such a curve and $f\in L^1(\HH^1\rest\Gamma)$,
then one defines
$$\|f\|_{\dot{B}_{p,p}^\alpha(\Gamma)}^p=
\iint_{(x,y)\in \Gamma^2}\frac{|f(x)-f(y)|^p}{|x-y|^{\alpha p + 1}} \,d\HH^1(x)\,d\HH^1(y).$$
The same definition applies if $\Gamma$ is a Lipschitz graph, say.
If $\gamma:S^1(0,r)\to\Gamma$ or $\gamma:\R\to\Gamma$ is a bilipschitz parameterization of $\Gamma$ (such as the arc
length parameterization), clearly we have
$$\|f\|_{\dot{B}_{p,q}^\alpha(\Gamma)}\approx \|f\circ \gamma\|_{\dot{B}_{p,q}^\alpha(S^1(0,r))},$$
for $f$ compactly supported. 

Concerning the Besov spaces of regularity $1<\alpha<2$, let us remark that, for $f\in L^1_{loc}(\R)$, 
\begin{equation}\label{eqn67}
\|f\|_{\dot{B}_{p,q}^\alpha}^p \approx \|f'\|_{\dot{B}_{p,q}^{\alpha-1}}^p,
\end{equation}
where $f'$ is the distributional derivative of $f$. Further we will use the following 
characterization in terms of the coefficients $\beta_1$ defined in \rf{defbeta}, due to Dorronsoro \cite[Theorems 1 and 2]{Dorronsoro2}. For $0<\alpha<2$ and $1\leq p,q<\infty$, one has:
$$\|f\|_{\dot{B}_{p,q}^\alpha} \approx \left(\int_0^\infty \left(h^{-\alpha+1} 
\|\beta_1(f,I(\cdot,h))\|_p\right)^q\,\frac{dh}h\right)^{1/q}.$$
Again, this comparability should be understood modulo polynomials, unless $f$ is compactly supported, say.
In the case $p=q$, an equivalent statement is the following:
$$\|f\|_{\dot{B}_{p,p}^\alpha}^p \approx \biggl(\sum_{I\in\DD(\R)}  
\biggl(\frac{\beta_1(f,I)}{\ell(I)^{\alpha-1}}\biggr)^p\,\ell(I)\biggr)^{1/p}.$$
For other indices $\alpha\geq
2$, there are analogous results which involve approximation by polynomials
of a fixed degree instead of affine functions.
Let us remark that the coefficients $\beta_1(f,I)$ are not introduced in \cite{Dorronsoro2}, and
instead a different notation is used.


\section{Auxiliary lemmas}\label{sec3}

\subsection{About Besov spaces}
\begin{lemma}\label{lemanorm}
Let $A:\R\to\R$ be a Lipschitz function with $\|A'\|_\infty\leq c_0$ and $\Gamma\subset\C$ its graph. Denote by $N_0(x)$ the unit
normal of $\Gamma$ at $(x,A(x))$ (whose vertical component is negative, say), which is defined a.e.
Then, 
\begin{equation}\label{eqn62}
|\Delta_h (A')(x)|\approx |\Delta_h N_0(x)|,
\end{equation}
with constants depending on $c_0$.
Thus, for $1\leq p<\infty$ and $0<\alpha<1$,
\begin{equation}\label{eqn63}
\|A\|_{\dot{B}_{p,p}^{\alpha+1}}\approx \|A'\|_{\dot{B}_{p,p}^{\alpha}}\approx\|N_0\|_{\dot{B}_{p,p}^{\alpha}},
\end{equation}
with constants depending on $\alpha$ and $p$, and also on $c_0$ in the second estimate.
\end{lemma}

Above, we set
$$\|N_0\|_{\dot{B}_{p,p}^{\alpha}}:= 
\|N_{0,1}\|_{\dot{B}_{p,p}^{\alpha}} + \|N_{0,2}\|_{\dot{B}_{p,p}^{\alpha}},$$
where $N_{0,i}$, $i=1,2$, are the components of $N_0$.

For the proof, see \cite{Cruz-Tolsa}.

\begin{rem}
As mentioned in Subsection \ref{subsec2.2}, from the characterization of Besov spaces in terms of differences, it turns
out that if $\gamma:\R\to\Gamma$ is an arc length parameterization of the Lipschitz graph $\Gamma$,
and $N(z)$ stands for the unit normal at $z\in\Gamma$ (with a suitable orientation), 
then
$$\|N\circ\gamma\|_{\dot{B}_{p,p}^{\alpha-1/p}}\approx\|N_0\|_{\dot{B}_{p,p}^{\alpha-1/p}}\approx\|N\|_{\dot{B}_{p,p}^{\alpha-1/p}(\Gamma)},$$
for $0<\alpha\leq1$ and $1<p<\infty$ such that $\alpha p>1$.

Recall that for a Lipschitz domain $\Omega$, whose boundary has an arc length parameterization
$\gamma:S^1(0,r_0)\to\partial \Omega$ (with $2\pi r_0=\HH^1(\partial \Omega)$), if $N(z)$ stands for
the outward unit normal at $z\in\partial\Omega$, we also have
$$\|N\|_{\dot{B}_{p,p}^{\alpha-1/p}(\partial\Omega)}\approx\|N\circ\gamma\|_{\dot{B}_{p,p}^{\alpha-1/p}
(S^1(0,r_0))}.$$
\end{rem}

\vv
In \rf{defbeta} we defined the coefficients $\beta_1$ associated to
a function $f$. Now we introduce an analogous notion replacing 
$f$ by a chord arc curve $\Gamma$ (which may be the boundary
of a Lipschitz domain). Given $P\in\DD(\Gamma)$, we set
\begin{equation}\label{defbetag} 
\beta_1(\Gamma,P) = \inf_L \frac1{\ell(P)}\int_{3P} \frac{\dist(x,L)}{\ell(P)}\,
d\HH^1(x),
\end{equation}
where the infimum is taken over all the lines $L\subset\C$.

Next lemma is a direct consequence of the previous results and the characterization of
homogeneous Besov spaces in terms of the $\beta_1$'s from Dorronsoro. For the detailed proof, see \cite{Cruz-Tolsa}.

\begin{lemma}\label{lemdorron}
Let $\Omega$ be a Lipschitz domain. Suppose that the outward unit normal satisfies
$N\in \dot{B}_{p,p}^\alpha(\partial\Omega)$, for some $1\leq p<\infty$, $0<\alpha<1$.
Then,
$$\sum_{P\in\DD(\partial\Omega)}  
\biggl(\frac{\beta_1(\partial\Omega,P)}{\ell(P)^{\alpha}}\biggr)^p\,\ell(P)
\lesssim \|N\|_{\dot{B}_{p,p}^\alpha(\partial\Omega)}^p + c\,\HH^1(\partial\Omega)^{1-\alpha\,p}.
$$
with $c$ depending on $\HH^1(\partial\Omega)/R$.
\end{lemma}

\begin{lemma}\label{lemaux1}
Consider functions $\vphi,f:\R\to\R$ and let $1\leq p<\infty$ and $0<\alpha<1$. We have
$$\|f\|_{\dot{B}_{p,p}^\alpha}^p  \lesssim
\iint_0^\infty \frac{|\vphi\,\Delta_h(f)|^p}{h^{\alpha p + 1}}\,dh\,dx
+ \|\vphi\|_{\dot{B}_{p,p}^\alpha}^p \,\|f\|_\infty^p.$$
\end{lemma}

\begin{proof}
Recall that
\begin{equation}\label{eqf43}
\|\vphi f\|_{\dot{B}_{p,p}^\alpha}^p  \approx
\iint_0^\infty \frac{|\Delta_h(\vphi f)|^p}{h^{\alpha p + 1}}\,dh\,dx.
\end{equation}
To prove the lemma, we just use that
$$\Delta_h(\vphi\,f)(x) = \vphi(x)\,\Delta_h(f)(x) + f(x+h)\,\Delta_h(\vphi)(x).$$
Then, plugging this identity into \rf{eqf43} the lemma follows easily.
\end{proof}

\begin{lemma}\label{lemaux2}
Let $A:\R\to\R$ be a Lipscthitz function supported on an interval $I$.
For $1\leq p<\infty$ and $0<\alpha<1$, we have
$$\|A\|_{\dot{B}_{p,p}^\alpha}\lesssim \|A'\|_\infty\,\ell(I)^{1-\alpha+1/p},$$
with a constant depending on $\alpha$ and $p$.
\end{lemma}

\begin{proof}
Denote $\ell=\ell(I)$. We set
\begin{align}\label{eqs99}
\|A\|_{\dot{B}_{p,p}^\alpha}^p & \approx 
\iint_0^\infty \frac{|\Delta_h(A)|^p}{h^{\alpha p}}\,\frac{dh}h\,dx\\
& = \iint_{0<h\leq\ell} \frac{|A(x+h)-A(x)|^p}{h^{\alpha p+1}}\,dh\,dx
+\iint_{h>\ell} \frac{|A(x+h)-A(x)|^p}{h^{\alpha p+1}}\,dh\,dx.\nonumber
\end{align}
For the first integral on the right side we use fact that the integrand vanishes
unless $x\in 3I$ and the Lipschitz condition 
$|A(x+h)-A(x)|\leq \|A'\|_\infty h$:
\begin{align*}
\iint_{0<h\leq\ell} \frac{|A(x+h)-A(x)|^p}{h^{\alpha p+1}}\,dh\,dx &
\leq \|A'\|_\infty^p\int_{x\in3I}\int_{h\leq\ell} \frac{h^p}{h^{\alpha p+1}}\,dh\,dx\\
& = 3\|A'\|_\infty^p\ell\int_{0<h\leq\ell} h^{p(1-\alpha)-1}\,dh \approx \|A'\|_\infty^p\,\ell^{p+1-\alpha p}.
\end{align*}

Concerning the last integral in \rf{eqs99}, we use again the fact that $A$ is supported
on $I$ and the estimate 
$|A(x+h)-A(x)|\leq \|A'\|_\infty\ell$:
\begin{multline*}
\iint_{h>\ell} \frac{|A(x+h)-A(x)|^p}{h^{\alpha p+1}}\,dh\,dx\\
 \leq
 \|A'\|_\infty^p\,\ell^p\left(\int_{x\in I}\int_{h>\ell} \frac1{h^{\alpha p+1}}\,dh\,dx
 + \int_{x\in I-h}\int_{h>\ell} \frac1{h^{\alpha p+1}}\,dh\,dx\right)\approx \|A'\|_\infty^p\,
 \ell^{p+1-\alpha p}.
 \end{multline*}
\end{proof}


\subsection{About the Beurling transform of $\chi_\Omega$}

Let $\Omega\subset\C$ be an open set. If $\Omega$ has finite Lebesgue measure, then
\begin{equation}\label{eqbeurling21}
B\chi_\Omega(z)=\lim_{\ve\to0}\frac{-1}\pi\int_{|z-w|>\ve}\frac{1}{(z-w)^2}
\chi_\Omega(w)\,dm(w).
\end{equation}
Otherwise, $B(\chi_\Omega)$ is a BMO function and,
thus, it is defined modulo constants. Actually, a possible way to assign a precise value to $B(\chi_
\Omega)(z)$
is the following:
\begin{equation}\label{eqbeurling22}
B\chi_\Omega(z)=\lim_{\ve\to0}\frac{-1}\pi\int_{|z-w|>\ve}\left(\frac{1}{(z-w)^2} - 
\frac1{(z_0-w)^2}\right)
\chi_\Omega(w)\,dm(w),
\end{equation}
where $z_0$ is some fixed point, with $z_0\not\in\overline\Omega$, for example. It is easy to check
that the preceding principal value integral exists for all $z\in\C$.

Moreover the following results hold:

\begin{lemma}\label{lemct1}
Let $\Omega\subset\C$ be an open set.
The function $B(\chi_\Omega)$ is analytic in $\C\setminus\partial\Omega$ and moreover, for every $z\in
\C\setminus\partial\Omega$, 
\begin{equation}\label{eq**}
\partial B(\chi_\Omega)(z) = \frac2\pi \int_{|z-w|>\ve}\frac1{(z-w)^3}\,\chi_{\Omega}(w)\,dm(w),
\end{equation}
for $0<\ve<\dist(z,\partial \Omega)$. 
\end{lemma}

When $\Omega$ has infinite measure, saying that $B(\chi_\Omega)$ is analytic in $\C\setminus\partial\Omega$ means that the function defined in \rf{eqbeurling22} is analytic for each choice of $z_0$.
Notice that, in any case, the derivative $\partial B(\chi_\Omega)$ is independent of $z_0$.

\begin{lemma}\label{lemmaPi}
Let $\Pi\subset\C$ be a half plane. Then 
$\partial B(\chi_\Pi)=0$ in $\C\setminus\partial\Pi$. Equivalently,
for all $z\not\in\partial\Pi$ and $0<\ve<\dist(z,\partial\Pi)$, we have
$$\int_{|z-w|>\ve}\frac1{(z-w)^3}\,\chi_{\Pi}(w)\,dm(w) = 0.$$
\end{lemma}

For the proofs of the preceding two lemmas, see \cite{Cruz-Tolsa}, for example.
By using very similar arguments, 
 for any $z\in \C\setminus\partial\Omega$, one gets
\begin{equation}\label{eq***}
\partial^2 B(\chi_\Omega)(z) = \frac{-6}\pi \int_{|z-w|>\ve}\frac1{(z-w)^4}\,\chi_{\Omega}(w)\,dm(w),
\end{equation}
for $0<\ve<\dist(z,\partial \Omega)$. The details are left for the reader.

\begin{lemma}\label{lemderiv}
Let $\Omega$ be either a Lipschitz or a special Lipschitz domain.
For all $w\in\C\setminus \partial\Omega$
and all $\ve$ with $0<\ve<\dist(w,\partial\Omega)$,
 we have
 \begin{equation}\label{eqtyu}
\partial B(\chi_\Omega)(w) = \frac i{2\pi} \int_{\partial\Omega} \frac1{(w-z)^2}\,d\overline z.
\end{equation}
\end{lemma}

\begin{proof}
In sense of distributions, we have
$$\partial B(\chi_\Omega) = \frac{-1}\pi\,\partial\, \Bigl({\rm p.v.}\,\frac1{z^2} * \chi_\Omega\Bigr).$$
Suppose that first that $\Omega$ is bounded. Then we have
\begin{equation}\label{eqdis7}
{\rm p.v.}\,\frac1{z^2} * \chi_\Omega = 
{\rm p.v.}\,\frac1{z^2} * \partial\chi_\Omega,
\end{equation}
It turns out that, in the sense of distributions, 
$$\partial \chi_\Omega = -\frac i2\,d\overline z\rest \partial\Omega.$$
Indeed, given $\vphi\in\CC_c^\infty(\R^2)$,
$$\langle \partial \chi_\Omega,\,\vphi\rangle = 
-\int_\Omega \partial\vphi\,dm =
-\frac i2\int_{\partial\Omega} \vphi\,d\overline z,$$
which proves our claim.

So we deduce that
$$\partial B(\chi_\Omega) = \frac{i}{2\pi}\,{\rm p.v.}\,\frac1{z^2} * d\overline z\rest \partial\Omega,$$
in the sense of distributions. From the first identity in \rf{eqtyu}, it is clear $\partial B(\chi_\Omega)$ is analytic in $\C\setminus\partial\Omega$. So the identity above holds pointwise in $\C\setminus \partial\Omega$.

For a special a Lipschitz domain we have to be a little careful, because both ${\rm p.v.}\,\dfrac1{z^2}$ and $\chi_\Omega$ are distributions with non compact support, and so the identity
\rf{eqdis7} is not clear. 

 Consider the upper half plane $\Pi=\{(x,y):y>0\}$.
Since $B(\chi_\Pi)$ is constant in $\C\setminus\R$,  $\partial B(\chi_\Pi)=0$ in $\C\setminus\R$, and so 
$$\partial B(\chi_\Omega) = \partial B(\chi_\Omega-\chi_\Pi) = 
\frac{-1}\pi\,\partial\, \Bigl({\rm p.v.}\,\frac1{z^2} * (\chi_\Omega-\chi_\Pi)\Bigr)
\quad \mbox{in $\C\setminus\R$}.$$
Now, observe that $\chi_\Omega-\chi_\Pi$ has compact support, because the Lipschitz function $A$ has
compact support, and thus
$$\Bigl({\rm p.v.}\,\frac1{z^2} * (\chi_\Omega-\chi_\Pi)\Bigr) = 
{\rm p.v.}\,\frac1{z^2} * \partial(\chi_\Omega-\chi_\Pi).$$
In the sense of distributions, 
$$\partial \chi_\Omega = -\frac i2\,d\overline z\rest \partial\Omega
\quad\mbox{and}\quad\partial \chi_\Pi = -\frac i2\,d\overline z\rest \partial\Pi.$$
As a consequence,
$$\partial B(\chi_\Omega-\chi_\Pi) = \frac i{2\pi}\,
{\rm p.v.}\,\frac1{z^2} * (d\overline z\rest \Omega - d\overline z\rest \partial\Pi).$$
Therefore,
$$\partial B(\chi_\Omega) = \frac i{2\pi}\,
{\rm p.v.}\,\frac1{z^2} * (d\overline z\rest \Omega - d\overline z\rest \partial\Pi)
\quad \mbox{in $\C\setminus\R$}.$$

Taking into account that $d\overline z\rest \partial\Pi= dz\rest \partial\Pi$ and, by Cauchy's formula,
$$\int_{\partial\Pi} \frac1{(w-z)^2}\,dz=0\quad \mbox{for all $w\in\C\setminus\R$,}$$
 we deduce that
$$\partial B(\chi_\Omega)(w)
= \frac i{2\pi} \int_{\partial\Omega} \frac1{(w-z)^2}\,d\overline z\quad \mbox{for all $w\in\C\setminus (\R\cup\partial\Omega)$.}$$
Since $\partial B(\chi_\Omega)$ is analytic in $\C\setminus\partial\Omega$, the identity above holds for all $w\in\C\setminus \partial\Omega$.
\end{proof}

The following result is a straightforward consequence of the preceding lemma.

\begin{lemma}\label{lemderiv2}
Let $\Omega=\{(x,y)\in\C:y=A(x)\}$ be a special Lipschitz domain (with $A$ Lipschitz and compactly supported). Then we have
$$\partial B(\chi_\Omega)(w) =  \frac1\pi\int_{\R} \frac{A'(x)}{(x+iA(x)-w)^2}\,dx,\qquad \mbox{for all 
$w\not\in
\partial\Omega.$}$$
\end{lemma}

\begin{proof}
By Cauchy's formula, it follows that
$$\int_{\partial\Omega} \frac1{(w-z)^2}\,dz=0.$$
From this fact and the preceding lemma, we infer that
\begin{equation}\label{eqbe1}
\partial B(\chi_\Omega)(w) = \frac i{2\pi} \int_{\partial\Omega} \frac1{(z-w)^2}\,(d\overline z- dz).
\end{equation}
The boundary $\partial\Omega$ can be parameterized by $\{x+iA(x):\,x\in\R\}$. Then we have $dz=(1+iA'(x))\,dx$ and
thus
$d\overline z- dz =-2i A'(x)\,dx$. Plugging this identity into \rf{eqbe1}, one concludes the lemma.
\end{proof}

\begin{lemma}\label{lemnou9}
Let $\Omega\subset\C$ be either a Lipschitz or a special Lipschitz domain. Then, for all $z\in Q$ with
$Q\in\WW(\Omega)$,
Then,
\begin{equation}\label{eqclaim0.5}
\bigl|\partial B\chi_\Omega(z)\bigr|\lesssim\sum_{R\in\DD(\partial\Omega):R\supset\phi(Q)}\frac{\beta_1(R)}{\ell(R)}+ \frac1{\diam(\Omega)},
\end{equation}
and
\begin{equation}\label{eqclaim1}
\bigl|\partial^2 B\chi_\Omega(z)\bigr|\lesssim \sum_{R\in\DD(\partial\Omega):R\supset\phi(Q)}\frac{\beta_1(R)}{\ell(R)^2}+ \frac1{\diam(\Omega)^2}.
\end{equation}
\end{lemma}

\begin{proof}
The estimate \rf{eqclaim0.5} has been proved in \cite{Cruz-Tolsa} (see equation (5.2) there). The proof of
the inequality \rf{eqclaim1} is very similar. For completeness, we sketch the arguments.
We may assume that $\beta_1(c_4\phi(Q))\leq\ve_0$, with $\ve_0>0$ small enough,
 for some fixed absolute constant
$c_4>10$, say. 
Indeed, from \rf{eq***} it turns out that $\bigl|\partial^2 B\chi_\Omega(z)\bigr|\leq c/\ell(Q)^2$, by choosing $\ve=\ell(Q)$ there, and so \rf{eqclaim1} holds if $\beta_1(c_4\phi(Q))>\ve_0$, with some constant depending on
$\ve_0$.

So suppose that $\beta_1(c_4\phi(Q))\leq\ve_0$, with $\ve_0$ very small. In this case, $L_Q$ 
is very close to 
$\partial\Omega$ near $\phi(Q)$, and then one infers that 
$\dist(z,L_Q)\approx\ell(Q).$
Denote by $\Pi_Q$ the half plane whose boundary is $L_Q$ and contains $z$. Take $0<\ve<\dist(z,\partial\Omega)$.
Since $(\frac1{z^4}\chi_{B(0,\ve)^c})*\chi_{\Pi_Q}$ vanishes on $\Pi_Q\ni z$ (because $\partial^2
B\chi_\Omega(z)=0$), we have
$$\bigl|\partial^2 B\chi_\Omega(z)\bigr|= \Bigl|\Bigl(\frac{6\pi}{z^4}\chi_{B(0,\ve)^c}\Bigr)*(\chi_\Omega- \chi_{\Pi_Q})(z)\Bigr|\leq 
\frac{6\pi}{|z|^3}*\chi_{\Omega\Delta\Pi_Q}(z).$$
For each $n\geq0$, let $B_n$ be a ball centered at $w'\in\phi(Q)$ with 
$$\diam(B_n)= 2^n\diam(\phi(Q))\approx 2^n\ell(Q),$$
and set also $B_{-1}=\varnothing$. 
For some $n_0$ such that $\diam(\Omega)\approx\diam(B_{n_0})$,  similarly to \cite{Cruz-Tolsa},
we have
\begin{align*}
\frac{6\pi}{|z|^4}*\chi_{\Omega\Delta\Pi_Q}(z) & = 
\sum_{n=0}^{n_0}\frac{6\pi}{|z|^4} * \chi_{B_n\cap(\Omega\Delta\Pi_Q)}(z) +
\frac{6\pi}{|z|^4} * 
\chi_{B_N^c\cap(\Omega\Delta\Pi_Q)}(z)\\
&\leq c
\sum_{n=0}^{n_0}\frac{1}{\ell(2^nQ)^4} \,m(B_n\cap(\Omega\Delta\Pi_Q))+ \frac{c}{\diam(\Omega)^2}.
\end{align*}
 By Lemma 4.3 from \cite{Cruz-Tolsa},  we have
$$m(B_n\cap(\Omega\Delta\Pi_Q))\leq c\sum_{P\in\DD(\partial \Omega):\phi(Q)\subset P\subset R}
\beta_1(P)\diam(R)^2,$$
where $R\in\DD(\partial\Omega)$ is some cube containing $\phi(Q)$ such that $\ell(R)
\approx\diam(B_n)$.
Then we obtain
\begin{align}\label{nucl}
\frac{6\pi}{|z|^4}*\chi_{\Omega\Delta\Pi_Q}(z) & \leq c
\sum_{\substack{R\in\DD(\partial\Omega):\\R\supset \phi(Q)}}\frac{1}{\ell(R)^4} \,
\sum_{\substack{P\in\DD(\partial \Omega):\\ \phi(Q)\subset P\subset R}}
\beta_1(P)\ell(R)^2 + \frac{c}{\diam(\Omega)^2}\\
& = c
\sum_{\substack{P\in\DD(\partial \Omega):\\P\supset \phi(Q)}}
\beta_1(P) 
\sum_{\substack{R\in\DD(\partial\Omega):\\R\supset P}}\frac1{\ell(R)^2}+ \frac{c}{\diam(\Omega)^2} \nonumber\\
&\leq c\sum_{\substack{P\in\DD(\partial\Omega):\\P\supset\phi(Q)}}\frac{\beta_1(P)}{\ell(P)^2}
+ \frac{c}{\diam(\Omega)^2} ,\nonumber
\end{align}
which proves \rf{eqclaim1}.
\end{proof}

Since $B\chi_\Omega$ is analytic in $\Omega$, it is clear that
$$\|B\chi_\Omega\|_{\dot W^{1,p}(\Omega)}^p\approx \int_\Omega |\partial B\chi_\Omega|^p\,dm.$$ 
On the other hand, for $0<\alpha<1$, we set $\|B\chi_\Omega\|_{\dot W^{\alpha,p}(\Omega)} =  \|D^\alpha B\chi_\Omega\|_{L^p(\Omega)},$
with $D^\alpha f$  defined in \rf{eqdalfa}. The following lemma relates $D^\alpha B\chi_\Omega$ to 
$\partial B\chi_\Omega$, and it will play a key role for the proof of Theorems \ref{teodom} and \ref{teopri}
in the case $0<\alpha<1$.

\begin{lemma}\label{lemkeyalfa}
Let 
$0<\alpha<1$ and $1<p<\infty$ be such that $\alpha p> 1$. For $0<\theta\leq 1$, if $\Omega$
is either a Lipschitz domain or a special Lipschitz domain, we have
\begin{equation}\label{eqdf890}
\|B(\chi_\Omega)\|_{\dot W^{\alpha,p}(\Omega)}^p\gtrsim \theta^{p-\alpha p} \!\!\int_\Omega|\partial B\chi_\Omega(z)|^p\,\dist(z,
\partial\Omega)^{p-\alpha p}\,dm(z) - c_3\,\theta^{2p-\alpha p}
\|N\|_{\dot B_{p,p}^{\alpha-1/p}(\partial\Omega)}^p,
\end{equation}
where the constant $c$ depends on $p$ and $\alpha$.
\end{lemma}

Notice that the integral on the right side is multiplied by by $\theta^{p-\alpha p}$,
while $\|N\|_{\dot B_{p,p}^{\alpha-1/p}(\partial\Omega)}^p$ by $\theta^{2p-\alpha p}$. The fact
that for $\theta\ll1$ we have
$\theta^{2p-\alpha p}\ll\theta^{p-\alpha p}$ will be important for the proof of Theorems \ref{teodom}
and \ref{teopri}.

\begin{proof}
For $x\in Q\in\WW(\Omega)$, we have
\begin{equation}\label{eqd303}
D^{\alpha}f(x)^2=\int_{\Omega}\frac{|f(x)-f(y)|^{2}}{|x-y|^{2+2\alpha}}\,dm(y)\geq 
\int_{|y-x|\leq \theta\ell(Q)}\frac{|f(x)-f(y)|^{2}}{|x-y|^{2+2\alpha}}\,dm(y).
\end{equation}
Let $f$ be analytic in $\Omega$, such as $B\chi_\Omega$. For $x\in Q\in\WW(\Omega)$ and $y\in\Omega$ such that
$$|x-y|\leq\theta \ell(Q),$$
 we have
$$|f(y)-f(x) - f'(x)(y-x)|\leq \frac12\,\sup_{w\in 3Q}|f''(w)|\,|x-y|^2.$$
Thus,
\begin{equation}\label{eqff55}
2|f(y)-f(x)|^2\geq |f'(x)(y-x)|^2 - \sup_{w\in 3Q}|f''(w)|^2\,|x-y|^4.
\end{equation}
Plugging this estimate into \rf{eqd303} yields
\begin{align*}
D^{\alpha}f(x)^2& \geq 
\int_{|y-x|\leq \theta\ell(Q)} \frac1{|x-y|^{2\alpha}}\,|f'(x)|^2\,dm(y)\\
&\quad -
\int_{|y-x|\leq \theta\ell(Q)} |x-y|^{2-2\alpha}\sup_{w\in 3Q}|f''(w)|^2\,dm(y)\\
&\gtrsim \theta^{2-2\alpha}\ell(Q)^{2-2\alpha} \,|f'(x)|^2- c\,\theta^{4-2\alpha}\ell(Q)^{4-2\alpha} \sup_{w\in 3Q}|f''(w)|^2.
\end{align*}
Therefore, since $\ell(Q)\approx\dist(x,\partial\Omega)$,
\begin{align*}
\|D^{\alpha}f\|_{L^p(\Omega)}^p &\gtrsim
\theta^{p-\alpha p} \int_\Omega |f'(x)|^p\,\dist(x,\partial\Omega)^{p-\alpha p}\,dm(x)\\&\quad -
c\,\theta^{2p-\alpha p}\sum_{Q\in\WW(\Omega)}
 \ell(Q)^{2+2p-\alpha p} \sup_{w\in 3Q}|f''(w)|^p.
\end{align*}

Hence, to prove \rf{eqdf890} it is enough to show that, for $f=B\chi_\Omega$,
\begin{equation}\label{eqdf891}
S:=\sum_{Q\in\WW(\Omega)}
 \ell(Q)^{2+2p-\alpha p} \sup_{w\in 3Q}|f''(w)|^p \lesssim \|N\|_{\dot B_{p,p}^{\alpha-1/p}(\partial\Omega)}^p.
 \end{equation}
Now, from Lemma \ref{lemnou9}, it turns out that for all $w\in3Q$, with $Q\in\WW(\Omega)$
$$
\bigl|\partial^2 B\chi_\Omega(w)\bigr|\lesssim\sum_{R\in\DD(\partial\Omega):R\supset\phi(Q)}\frac{\beta_1(R)}{\ell(R)^2}+ \frac1{\diam(\Omega)^2}.$$
Then we infer the term $S$ in \rf{eqdf891} satisfies
$$S\lesssim\sum_{Q\in\WW(\Omega)}
 \ell(Q)^{2+2p-\alpha p} \biggl(\sum_{R\in\DD(\partial\Omega):R\supset\phi(Q)}\frac{\beta_1(R)}{\ell(R)^2}\biggr)^p+ \frac{m(\Omega)}{\diam(\Omega)^{\alpha p}}.$$
The last term on the right side is bounded by $\diam(\Omega)^{2-\alpha p}$. For the first one we use Cauchy-Schwarz, and then we 
get
\begin{align*}
\biggl(
\sum_{R\in\DD(\partial\Omega):R
R\supset\phi(Q)}\!\!\frac{\beta_1(R)}{\ell(R)^2} \biggr)^p\! &
\leq \biggl(
\sum_{R\in\DD(\partial\Omega):R\supset\phi(Q)}\!\frac{\beta_1(R)^p}{\ell(R)^{2p-1/2}} \!\biggr)
\biggl(
\sum_{R\in\DD(\partial\Omega):R\supset\phi(Q)}\!\frac{1}{\ell(R)^{p'/(2p)}} \biggr)^{p/p'}
\\
&\lesssim
\sum_{R\in\DD(\partial\Omega):R\supset\phi(Q)}\frac{\beta_1(R)^p}{\ell(R)^{2p-1/2}} \frac1{\ell(\phi(Q))^{1/2}}.
\end{align*} 
 Thus,
\begin{align*}
\sum_{Q\in\WW(\Omega)}
 \ell(Q)^{2+2p-\alpha p} \biggl( &\sum_{R\in\DD(\partial\Omega):R\supset\phi(Q)} \frac{\beta_1(R)}{\ell(R)^2}\biggr)^p \\
 &\lesssim
\sum_{Q\in \WW(\Omega)}
\sum_{P\in\DD(\partial\Omega):P\supset\phi(Q)}\frac{\beta_1(P)^p}{\ell(P)^{2p-1/2}}\, \ell(\phi(Q))^{3/2+2p-\alpha p}\\
& = 
\sum_{P\in\DD(\partial\Omega)}\frac{\beta_1(P)^p}{\ell(P)^{2p-1/2}}\, \sum_{Q\in \WW(\Omega):\phi(Q)\subset P}
\ell(\phi(Q))^{3/2+2p-\alpha p}.
\end{align*}
 Notice that
$$\sum_{Q\in \WW(\Omega):\phi(Q)\subset P}
\ell(\phi(Q))^{3/2+2p-\alpha p}\lesssim\sum_{\wt Q\in \DD(\partial\Omega):\wt Q\subset P}
\ell\bigl(\wt Q\bigr)^{3/2+2p-\alpha p}\lesssim\ell(P)^{3/2+2p-\alpha p},$$
because $3/2+2p-\alpha p>1$. Hence,
$$\sum_{Q\in\WW(\Omega)}
 \ell(Q)^{2+2p-\alpha p} \biggl(\sum_{R\in\DD(\partial\Omega):R\supset\phi(Q)} \frac{\beta_1(R)}{\ell(R)^2}\biggr)^p\lesssim
\sum_{P\in\DD(\partial\Omega)}\beta_1(P)^p\,\ell(P)^{2 -\alpha p}.$$

Therefore,
$$S\lesssim \sum_{P\in\DD(\partial\Omega)}\biggl(\frac{\beta_1(P)}{\ell(P)^{\alpha+1/p}}\biggr)^p\,\ell(P) +
\diam(\Omega)^{2-\alpha p}\approx \|N\|_{\dot{B}_{p,p}^{\alpha-1/p}(\partial\Omega)}^p,$$
by Lemma \ref{lemdorron}, as wished.

The arguments for special Lipschitz domains
are analogous, and even easier. Roughly speaking, the only difference is that the terms above 
which involve $\diam(\Omega)$ do not appear.
\end{proof}


\section{The main lemma and the proof of Theorem \ref{teopri}}\label{sec4}

The main result of this section is the following.

\begin{lemma}[{\bf Main Lemma}]\label{mainlemma}
Let $\Omega\subset\C$ be a special $\delta$-Lipschitz domain. Let $1<p<\infty$ and $0<\alpha\leq1$ be such
that $\alpha p>1$. If $\delta$ is small enough, then
\begin{equation}\label{eqmain73}
\int_\Omega|\partial B\chi_\Omega(z)|^p\,\dist(z,
\partial\Omega)^{p-\alpha p}\,dm(z) \gtrsim \|N\|_{\dot B_{p,p}^{\alpha-1/p}(\partial\Omega)}^p.
\end{equation}
\end{lemma}

Before worrying about the proof of the preceding result we show that this yields Theorem \ref{teopri} as
an easy consequence.

\begin{proof}[\bf Proof of Theorem \ref{teopri}]
In the case $\alpha=1$, it is clear that
$$\int_\Omega|\partial B\chi_\Omega(z)|^p\,\dist(z,
\partial\Omega)^{p-\alpha p}\,dm(z) \approx \|B\chi_\Omega\|_{\dot W^{1,p}(\Omega)}^p,$$
and thus the theorem is a straightforward consequence of \rf{eqmain73}.
For $0<\alpha<1$, we need to use Lemma \ref{lemkeyalfa} too. Indeed, if $\theta$ is chosen small enough,
from the Main Lemma, we will have
$$\theta^{p-\alpha p} \int_\Omega|\partial B\chi_\Omega(z)|^p\,\dist(z,
\partial\Omega)^{p-\alpha p}\,dm(z) \geq 2 c_3\,\theta^{2p-\alpha p}
\|N\|_{B_{p,p}^{\alpha-1/p}(\partial\Omega)}^p,$$
where $c_3$ is the constant appearing in \rf{eqdf890}. Then Lemma \ref{lemkeyalfa} tells us that
\begin{equation}\label{eqcor5}
\|B(\chi_\Omega)\|_{\dot W^{\alpha,p}(\Omega)}^p\gtrsim \theta^{p-\alpha p} \!\!\int_\Omega|\partial B\chi_\Omega(z)|^p\,\dist(z,
\partial\Omega)^{p-\alpha p}\,dm(z).
\end{equation}
Together with the Main Lemma again this implies that
$$\|B(\chi_\Omega)\|_{\dot W^{\alpha,p}(\Omega)}^p\gtrsim \|N\|_{\dot B_{p,p}^{\alpha-1/p}(\partial\Omega)}^p.
$$
\end{proof}

Notice also that from the Main Lemma, the inequality \rf{eqcor5}, and the fact that 
$$\|B(\chi_\Omega)\|_{\dot W^{\alpha,p}(\Omega)}\lesssim \|N\|_{\dot B_{p,p}^{\alpha-1/p}(\partial\Omega)},$$
proved in \cite{Cruz-Tolsa}, we deduce the following.

\begin{coro}\label{coro40}
Let $\Omega\subset\C$ be a special $\delta$-Lipschitz domain. Let $1<p<\infty$ and $0<\alpha<1$ be such
that $\alpha p>1$. If $\delta$ is small enough, then
$$\|B(\chi_\Omega)\|_{\dot W^{\alpha,p}(\Omega)}^p\approx
\int_\Omega|\partial B\chi_\Omega(z)|^p\,\dist(z,
\partial\Omega)^{p-\alpha p}\,dm(z) \approx \|N\|_{\dot B_{p,p}^{\alpha-1/p}(\partial\Omega)}^p.
$$
\end{coro}

The remaining of this section is devoted to the proof of the Main Lemma. First some remarks
abut notation and terminology: recall that in Subsection \ref{subsec2.1} to each square $Q\in\WW(\Omega)$ we assigned a cube
$\phi(Q)\in\DD(\partial\Omega)$ with diameter and distance to $Q$ both comparable to $\ell(Q)$.
In the case of special Lipschitz domains the following precise definition of $\phi(Q)$ is very convenient.
Given a square $Q=(a,b]\times(c,d]\in\WW(\Omega)$, we consider the arc 
$$\phi(Q)= \{(x,A(x)):\,a< x\leq b\}.$$
In particular, notice that $\phi(Q)\in\DD(\partial\Omega)$.
Observe also that $\ell(Q)=\HH^1(\phi(Q))$ and that 
$$\dist(Q,\phi(Q))\approx\ell(Q).$$
We denote 
$$\ell(\phi(Q)):=\ell(Q).$$
Moreover, given $a>1$ and $P\in\DD(\Omega)$ of the form
$$P=\{(x,A(x)):x\in I\},$$
for some interval $I\subset \R$, we let $aP$ be the following arc from $\partial\Omega$:
$$P=\{(x,A(x)):x\in aI\}.$$

\begin{lemma}\label{lemderiv3}
Let $\Omega$ be as in Theorem \ref{teopri}. Consider a square $Q\in\WW(\Omega)$ and denote by
$L_Q$ a line that minimizes $\beta_1(\phi(Q))$. Let $y=g_Q(x)$ the affine map defining $L_Q$. Then, for
any $w\in 3Q$ we have
$$
\left|\imag\left(\partial B\chi_\Omega(w) - \frac1\pi\int_{\R} \frac{A'(x)}{(x+i\,g_Q(x)-w)^2}\,dx
\right)\right|\leq c\,\delta\sum_{P\in\DD(\partial\Omega):P\supset\phi(Q)} \frac{\beta_1(P)}{\ell(P)}\,.$$
\end{lemma}

\begin{proof}
Let $w,Q,L_Q,g_Q$ be as in the statement above.
By the preceding lemma,
\begin{align}
\imag &\biggl(\partial B\chi_\Omega(w) - \frac1\pi\int_{\R} \frac{A'(x)}{(x+i\,g_Q(x)-w)^2}\,dx \biggr)\label{eqbeu1}\\
 & =\frac1\pi
\int_{\R} \imag\left(\frac{1}{(x+iA(x)-w)^2} - \frac{1}{(x+i\,g_Q(x)-w)^2}\right)A'(x)\,dx.\nonumber
\end{align}
Denoting $w=a+ib$, we have
\begin{align}
\imag&\left(\frac{1}{(x+iA(x)-w)^2} - \frac{1}{(x+i\,g_Q(x)-w)^2}\right)  \label{eqim1}\\
& =
-2(x-a)\left(\frac{A(x)-b}{\bigl((x-a)^2+(A(x)-b)^2\bigr)^2} - 
\frac{g_Q(x)-b}{\bigl((x-a)^2+(g_Q(x)-b)^2\bigr)^2}\right).\nonumber
\end{align}
We write the expression on the right inside the big parentheses as follows:
\begin{align*}
&\!\!\!\!\frac{\bigl(A(x)-b\bigr) - \bigl(g_Q(x)-b\bigr)}{\bigl((x-a)^2+(A(x)-b)^2\bigr)^2}\\& \quad+
\bigl(g_Q(x)-b\bigr)\left(\frac1{\bigl((x-a)^2+(A(x)-b)^2\bigr)^2}-
\frac1{\bigl((x-a)^2+(g_Q(x)-b)^2\bigr)^2}\right)\\
& =:T_1 + T_2.
\end{align*}
The first term equals
\begin{equation}\label{eqt1}
T_1 = \frac{A(x)-g_Q(x)}{|x+iA(x)-w|^4}.
\end{equation}
Concerning $T_2$, we have
\begin{align*}
&\frac1{\bigl((x-a)^2+(A(x)-b)^2\bigr)^2}-
\frac1{\bigl((x-a)^2+(g_Q(x)-b)^2\bigr)^2} \\
& = 
\frac{\bigl[2(x-a)^2+(g_Q(x)-b)^2+ (A(x)-b)^2\bigr]\,
\bigl[(g_Q(x)-b)^2 -(A(x)-b)^2\bigr]}
{\bigl((x-a)^2+(A(x)-b)^2\bigr)^2\,\bigl((x-a)^2+(g_Q(x)-b)^2\bigr)^2}\\
& = 
\frac{\bigl[2(x-a)^2+(g_Q(x)-b)^2+ (A(x)-b)^2\bigr]\,
\bigl[g_Q(x)+ A(x)-2b\bigr] \,
\bigl[g_Q(x)-A(x)\bigr]}
{\bigl((x-a)^2+(A(x)-b)^2\bigr)^2\,\bigl((x-a)^2+(g_Q(x)-b)^2\bigr)^2}.
\end{align*}
Notice now that 
$$(x-a)^2+(g_Q(x)-b)^2\approx(x-a)^2+(A(x)-b)^2.$$
 This follows easily from the fact that $A$ is a Lipschitz
graph with small slope and so we can assume that the slope of $g_Q$ is small and bounded independently of $Q$. Then, from
the last calculation, we obtain
$$|T_2|\lesssim 
\frac{
\bigl|A(x)-g_Q(x)\bigr|}
{\bigl((x-a)^2+(A(x)-b)^2\bigr)^2} = \frac{|A(x)-g_Q(x)|}{|x+iA(x)-w|^4}.$$
From \rf{eqbeu1}, \rf{eqim1}, \rf{eqt1}, the last estimate, and the fact that $\|A'\|_\infty\leq \delta$, we deduce that
\begin{align}\label{eqint3}
\left|\imag\left(\partial B\chi_\Omega(w) - \int_{\R} \frac{A'(x)}{(x+i\,g_Q(x)-w)^2}\,dx
\right)\right|&\lesssim
\int_{\R} 
\frac{|x-a|\,|A(x)-g_Q(x)|}{|x+iA(x)-w|^4}\,|A'(x)|
\,dx\\
&\lesssim\delta \int_{\R} 
\frac{|A(x)-g_Q(x)|}{|x+iA(x)-w|^3}
\,dx.\nonumber
\end{align}

Now we wish to estimate the last integral in \rf{eqint3}. To this end, we set
\begin{equation}\label{eqcur3}
\int_{\R} 
\frac{|A(x)-g_Q(x)|}{|x+iA(x)-w|^3}
\,dx \lesssim 
\sum_{k\geq0}\int_{ |x-a|\leq2^k\ell(Q)} 
\frac{\dist(x,L_Q)}{(2^k\ell(Q))^3}
\,dx.
\end{equation}

Consider $R\in\DD(\partial\Omega)$ such that $R\supset \phi(Q)$. 
Let $L_R$ be a line that minimizes $\beta_1(R)$. 
Then, as shown in \cite{Cruz-Tolsa},
$$\dist_H(L_Q\cap B(a,5\ell(R)),\,L_R\cap B(a,5\ell(R)))\leq c\sum_{P\in\DD(\partial \Omega):Q\subset P\subset R}
\beta_1(P)\,\ell(R),$$
where $\dist_H$ stands for the Hausdorff distance.
As a consequence,
for $x\in 3R$,
$$\dist(x,L_Q)= \dist(x,L_Q\cap B(a,5\ell(R)))\lesssim \dist(x,L_R) + 
\sum_{P\in\DD(\partial \Omega):Q\subset P\subset R}\beta_1(P)\,\ell(R).$$
Plugging this estimate into \rf{eqcur3}, we get
\begin{align*}
\int_{\R} 
\frac{|A(x)-g_Q(x)|}{|x+iA(x)-w|^3}
\,dx & \lesssim 
\sum_{R\in\DD(\partial \Omega):R\supset \phi(Q)}\int_{p_1(3R)}
\frac{\dist(x,L_Q)}{\ell(R)^3}
\,dx \\
& \lesssim
\sum_{\substack{R\in\DD(\partial \Omega):\\ R\supset \phi(Q)}}
\int_{p_1(3R)}
\frac{\dist(x,L_R)}{\ell(R)^3}
\,dx + \sum_{\substack{R\in\DD(\partial \Omega):\\ R\supset \phi(Q)}}\,\,
\sum_{\substack{P\in\DD(\partial \Omega):\\
\phi(Q)\subset P\subset R}}
\frac{\beta_1(P)}{\ell(R)}
\\ &
\lesssim \sum_{\substack{R\in\DD(\partial \Omega):\\ R\supset \phi(Q)}}\,\,
\sum_{\substack{P\in\DD(\partial \Omega):\\
\phi(Q)\subset P\subset R}}
\frac{\beta_1(P)}{\ell(R)} 
\lesssim
\sum_{\substack{P\in\DD(\partial \Omega):\\P\supset 
\phi(Q)}}
\frac{\beta_1(P)}{\ell(P)}.
\end{align*}
Together with \rf{eqint3}, this proves the lemma.
\end{proof}

\begin{lemma}\label{lemderiv4}
Let $\Omega$ be as in Theorem \ref{teopri} and $g_Q$ as in Lemma \ref{lemderiv3}. We have
\begin{multline}\label{eqdew11}
\sum_{Q\in\WW(\Omega)}\int_{3Q}
\left|\imag\left(\partial B\chi_\Omega(w) - \frac1\pi\int_{\R} \frac{A'(x)}{(x+i\,g_Q(x)-w)^2}\,dx
\right)\right|^p\ell(Q)^{p-\alpha p}\,dm(w)\\ \leq c\,\delta^p\,\|A\|_{\dot{B}_{p,p}^{1+\alpha-1/p}}^p.
\end{multline}
\end{lemma}

\begin{proof}
From Lemma \ref{lemderiv3} we infer that
\begin{multline*}
\sum_{Q\in\WW(\Omega)}\int_{3Q}
\left|\imag\left(\partial B\chi_\Omega(w) - \frac1\pi\int_{\R} \frac{A'(x)}{(x+i\,g_Q(x)-w)^2}\,dx
\right)\right|^p\ell(Q)^{p-\alpha p}\,dm(w)\\
\lesssim c\,\delta^p\,\sum_{Q\in\WW(\Omega)}\left(
\sum_{P\in\DD(\partial\Omega):P\supset\phi(Q)} \frac{\beta_1(P)}{\ell(P)}\right)^p\ell(Q)^{2+p-\alpha p}.
\end{multline*}

By Cauchy-Schwartz we have
\begin{align*}
\biggl(
\sum_{P\in\DD(\partial\Omega):P\supset\phi(Q)}\frac{\beta_1(P)}{\ell(P)} \biggr)^p &
\leq \biggl(
\sum_{P\in\DD(\partial\Omega):P\supset\phi(Q)}\frac{\beta_1(P)^p}{\ell(P)^{p-\frac12}} \biggr)
\biggl(
\sum_{P\in\DD(\partial\Omega):P\supset\phi(Q)}\frac{1}{\ell(P)^{\frac{p'}{2p}}} \biggr)^{\frac p{p'}}
\\
&\leq c
\sum_{P\in\DD(\partial\Omega):P\supset\phi(Q)}\frac{\beta_1(P)^p}{\ell(P)^{p-\frac12}} \frac1{\ell(\phi(Q))^{1/2}}.
\end{align*}
Thus,
\begin{align*}
\sum_{Q\in \WW(\Omega)}\biggl(
\sum_{\substack{P\in\DD(\partial\Omega):\\P\supset\phi(Q)}}\frac{\beta_1(P)}{\ell(P)} \biggr)^p\ell(Q)^{p-\alpha p} &\lesssim
\sum_{Q\in \WW(\Omega)}
\sum_{\substack{P\in\DD(\partial\Omega):\\P\supset\phi(Q)}}\frac{\beta_1(P)^p}{\ell(P)^{p-\frac12}}\, \ell(\phi(Q))^{3/2+p-\alpha p}\\
& = 
\sum_{P\in\DD(\partial\Omega)}\frac{\beta_1(P)^p}{\ell(P)^{p-\frac12}}\, \sum_{\substack{Q\in \WW(\Omega):\\\phi(Q)\subset P}}
\ell(\phi(Q))^{3/2+p-\alpha p}.
\end{align*}
Notice that
$$\sum_{Q\in \WW(\Omega):\phi(Q)\subset P}
\ell(\phi(Q))^{3/2+p-\alpha p}\leq c\sum_{\wt Q\in \DD(\partial\Omega):\wt Q\subset P}
\ell\bigl(\wt Q\bigr)^{3/2+p-\alpha p}\leq c\,\ell(P)^{3/2+p-\alpha p},$$
and so, the left side of \rf{eqdew11} is bounded above by
$$
c\,\delta^p\sum_{P\in\DD(\partial\Omega)}\frac{\beta_1(P)^p}{\ell(P)^{\alpha p-2}}.$$
Observe now that the last sum can be written as
$$\sum_{P\in\DD(\partial\Omega)}\frac{\beta_1(P)^p}{\ell(P)^{\alpha p-2}} = \sum_{P\in\DD(\partial\Omega)}\biggl(\frac{\beta_1(P)}{\ell(P)^{\alpha-\frac{1}p}}\biggr)^p\,\ell(P).$$
By \cite[Theorems 1 and 2]{Dorronsoro2}, this is comparable to $\|A\|_{\dot{B}_{p,p}^{1+\alpha-1/p}}^p
\approx \|N\|_{\dot{B}_{p,p}^{\alpha-1/p}}^p$, and so we are done.
\end{proof}

As a corollary of the preceding lemma, using the finite overlapping
of the squares $\{3Q\}_{Q\in\WW(\Omega)}$, we deduce that
\begin{align}\label{eqde56}
\int_\Omega &|\partial B\chi_\Omega(w)|^p\dist(w,
\partial\Omega)^{p-\alpha p}\,dm(w)  
\gtrsim \sum_{Q\in\WW(\Omega)}\int_{3Q}|\imag(\partial B\chi_\Omega)|^p\ell(Q)^{p-\alpha p}\,dm(w)
\\
&\gtrsim
\sum_{Q\in\WW(\Omega)}\int_{3Q}
\left|\int_{\R} \frac{A'(x)}{(x+i\,g_Q(x)-w)^2}\,dx\right|^p\ell(Q)^{p-\alpha p}\,dm(w) \nonumber\\
&\quad -
c\sum_{Q\in\WW(\Omega)}\int_{3Q}
\left|\imag\left(\partial B\chi_\Omega(w) - \int_{\R} \frac{A'(x)}{(x+i\,g_Q(x)-w)^2}\,dx
\right)\right|^p\ell(Q)^{p-\alpha p}\,dm(w)\nonumber\\
& \geq \sum_{Q\in\WW(\Omega)}\int_{3Q}
\left|\int_{\R} \frac{A'(x)}{(x+i\,g_Q(x)-w)^2}\,dx\right|^p\ell(Q)^{p-\alpha p}\,dm(w)- c\,\delta^p\,\|A\|_{\dot{B}_{p,p}^{1+\alpha-1/p}}^p.
\nonumber
\end{align}

We will prove below that
$$\sum_{Q\in\WW(\Omega)}\int_{3Q}
\left|\int_{\R} \imag\left(\frac{A'(x)}{(x+i\,g_Q(x)-w)^2}\right)\,dx
\right|^p\ell(Q)^{p-\alpha p}\,dm(w)\gtrsim \|A\|_{\dot{B}_{p,p}^{1+\alpha-1/p}}^p.$$
Together with \rf{eqde56}, this will yield the Main Lemma \ref{mainlemma}, by taking $\delta$ small enough.

For a given $w=a+bi\in\Omega$ and $Q\in\WW(\Omega)$ with $3Q\ni w$, we have
\begin{equation}\label{eqexp4}
\imag\left(\frac1{(x+i\,g_Q(x)-w)^2}\right) = 
\frac{-2(x-a)(g_Q(x)-b)}{\bigl((x-a)^2 + (g_Q(x)-b)^2\bigr)^2}.
\end{equation}
Now we denote
$$s_Q(w) = b - g_Q(a).$$
Observe that $s_Q(w)\approx \dist(w,L_Q)\approx\dist(w,\partial\Omega).$
Now, for some number $|\theta_Q|\lesssim\delta$,
$$g_Q(x) = \theta_Q(x-a) + g_Q(a) = \theta_Q(x-a) + b -s_Q(w).$$
Writing $\theta$ instead of $\theta_Q$ and $s$ instead of $s_Q(w)$ to simplify notation, the expression in \rf{eqexp4} equals
\begin{align}\label{eqdfe4}
\frac{2(x-a)(s-\theta(x-a))}{\bigl((x-a)^2 + (\theta(x-a)-s)^2\bigr)^2} = \frac2{(1+\theta^2)^2}\,
\frac{s(x-a)-\theta(x-a)^2}{\bigl[\bigl(x-a-\frac{\theta s}{1+\theta^2}\bigr)^2 + \left(\frac s{1+\theta^2}\right)^2\bigr]^2}
\end{align}
Notice that, with the change of variables $y=x-a$, $t=s/(1+\theta^2)$, the denominator in the last fraction can be written as
$\bigl[(y-\theta\,t)^2+t^2\Bigr]^2,$
while the numerator equals
$$
(1+\theta^2)\,t\, y-\theta\,y^2 =
(t-\theta^2t)(y-\theta\,t) + \theta\bigl[t^2-(y-\theta\,t)^2\bigr].
$$
So the last fraction on the right side of \rf{eqdfe4} equals
\begin{equation}\label{eqdfe5}
\frac{(t-\theta^2t)(y-\theta\,t)}{\bigl[(y-\theta\,t)^2+t^2\bigr]^2}
+\theta\frac{t^2-(y-\theta\,t)^2}{\bigl[(y-\theta\,t)^2+t^2\bigr]^2}.
\end{equation}

From \rf{eqexp4}, \rf{eqdfe4}, and \rf{eqdfe5}, using that $|\theta|\leq \|A'\|_\infty\leq \delta\ll 1$ and recalling that
$\theta$ depends on $Q\ni w$, but not on $x$,
we obtain
\begin{align*}
\left|\int_{\R} \imag\left(\frac{A'(x)}{(x+i\,g_Q(x)-w)^2}\right)\,dx
\right|& \geq
\biggl|\int_{\R} \frac{\frac s{1+\theta^2}\,\left(x-a-\frac{\theta s}{1+\theta^2}\right)}{\bigl[\bigl(x-a-\frac{\theta s}{1+\theta^2}\bigr)^2 + \left(\frac s{1+\theta^2}\right)^2\bigr]^2}\,A'(x)
\,dx
\biggr|\\
&\quad
- c\,\delta\,
\biggl|\int_{\R} \frac{\left(\frac s{1+\theta^2}\right)^2-\left(x-a-\frac{\theta s}{1+\theta^2}\right)^2}{\bigl[\bigl(x-a-\frac{\theta s}{1+\theta^2}\bigr)^2 + \left(\frac s{1+\theta^2}\right)^2\bigr]^2}\,A'(x)
\,dx\biggr|\\
& =: I_1(w,Q) - c\,\delta\,I_2(w,Q).
\end{align*}
Therefore,
\begin{align}\label{eqde59}
& \sum_{Q\in\WW(\Omega)}  \int_{3Q}
\left|\int_{\R} \imag\left(\frac{A'(x)}{(x+i\,g_Q(x)-w)^2}\right)\,dx
\right|^p\ell(Q)^{p-\alpha p}\,dm(w)\\
& \quad\gtrsim \sum_{Q\in\WW(\Omega)}\int_{3Q}I_1(w,Q)^p \ell(Q)^{p-\alpha p}\,dm(w) - c\delta^p\sum_{Q\in\WW(\Omega)}\int_{3Q}I_2(w,Q)^p\ell(Q)^{p-\alpha p}\, dm(w)
.\nonumber
\end{align}

\begin{lemma}\label{lem5}
Under the assumptions and notation above,
$$\sum_{Q\in\WW(\Omega)}\int_{3Q}I_1(w,Q)^p \ell(Q)^{p-\alpha p}\,dm(w) \gtrsim 
\int_0^\infty 
\|t^{\frac1p-1-\alpha}\,\psi_t * A\|_p^p \,\frac{dt}t,$$
where
$$\psi(x) = \frac{3x^2-1}{(x^2+1)^3}$$
and $\psi_t(x)=t^{-1}\psi(t^{-1}x)$.
Moreover, 
\begin{equation}\label{eqdr521}
\int_0^\infty 
\|t^{\frac1p-1-\alpha}\,\psi_t * A\|_p^p \,\frac{dt}t\gtrsim
\|A\|_{\dot{B}_{p,p}^{1+\alpha-1/p}}^p.
\end{equation}
\end{lemma}

\begin{proof}
Fix a square $Q\in\WW(\Omega)$, and set $Q=(a_1,a_2]\times (b_1,b_2]$.
Recall that, for $w=a+ib\in3Q$,
$$I_1(w,Q)=\biggl|\int_{\R} \frac{\frac s{1+\theta^2}\,\left(x-a-\frac{\theta s}{1+\theta^2}\right)}{\bigl[\bigl(x-a-\frac{\theta s}{1+\theta^2}\bigr)^2 + \left(\frac s{1+\theta^2}\right)^2\bigr]^2}\,A'(x)
\,dx\biggr|,$$
where $s\equiv s_Q(w) = b - g_Q(a),$ and 
$\theta\equiv\theta_Q$ is the slope of the affine line defined by $g_Q$, which approximates $\partial\Omega\cap p_1(3Q)$.
For $t>0$, consider the kernel 
$$K_t(x) = \frac{t\,x}{[x^2 + t^2]^2}.$$
Observe that
$$I_1(w,Q)= \Bigl|K_{\frac s{1+\theta^2}} * A'\,\Bigl(a+\frac{\theta s}{1+\theta^2}\Bigr)\Bigr|.$$

We have
$$\int_{3Q}I_1(w,Q)^p \,dm(w) = \int_{a_1-\ell(Q)}^{a_2+\ell(Q)} \int_{t_1-\ell(Q)}^{t_2+\ell(Q)}
\Bigl|K_{\frac s{1+\theta^2}} * A'\,\Bigl(x+\frac{\theta s}{1+\theta^2}\Bigr)\Bigr|^p \,dt\,dx.$$
Observe now that, assuming $\delta$ small enough, for each $x\in p_1(3Q)$ we have
\begin{align*}
\int_{t_1-\ell(Q)}^{t_2+\ell(Q)}
\Bigl|K_{\frac s{1+\theta^2}} * A'\,\Bigl(x+\frac{\theta s}{1+\theta^2}\Bigr)\Bigr|^p \,dt 
&= 
\int_{t_1-\ell(Q)}^{t_2+\ell(Q)}
\Bigl|K_{\frac{t-g_Q(x)}{1+\theta^2}} * A'\,\Bigl(x+\frac{\theta s}{1+\theta^2}\Bigr)\Bigr|^p \,dt 
\\
&\geq
\int_{t_1-\frac12\ell(Q)}^{t_2+\frac12\ell(Q)}
\Bigl|K_{t-A(a_1)} * A'\,\Bigl(x+\frac{\theta s}{1+\theta^2}\Bigr)\Bigr|^p \,dt 
.
\end{align*}
Also, it follows easily  that 
\begin{align*}
\int_{a_1-\ell(Q)}^{a_2+\ell(Q)} 
\int_{t_1-\frac12\ell(Q)}^{t_2+\frac12\ell(Q)}&
\Bigl|K_{t-A(a_1)} * A'\,\Bigl(x+\frac{\theta s}{1+\theta^2}\Bigr)\Bigr|^p \,dt\,dx\\
&\geq 
\int_{a_1}^{a_2} 
\int_{t_1-\frac12\ell(Q)}^{t_2+\frac12\ell(Q)}
\bigl|K_{t-A(a_1)} * A'\,(x)\bigr|^p \,dt\,dx\\
& \geq
\int_{a_1}^{a_2} 
\int_{t_1}^{t_2}
\bigl|K_{t-A(x)} * A'\,(x)\bigr|^p \,dt\,dx.
\end{align*}
As a consequence,
\begin{align}\label{eqdr5}
\sum_{Q\in\WW(\Omega)}\int_{3Q}I_1(w,Q)^p \ell(Q)^{p-\alpha p}\,dm(w)
&\gtrsim \sum_{Q\in\WW(\Omega)} \iint_{(x,t)\in Q} 
\bigl|K_{t-A(x)} * A'\,(x)\bigr|^p t^{p-\alpha p}\,dt\,dx \\
& = \iint_{(x,t)\in \Omega} 
\bigl|K_{t-A(x)} * A'\,(x)\bigr|^p t^{p-\alpha p}\,dt\,dx \nonumber\\
& = \int_{x\in\R}\int_0^\infty 
\bigl|K_{t} * A'\,(x)\bigr|^p t^{p-\alpha p}\,dt\,dx.\nonumber
\end{align}

Notice now that $K_{t} * A' = (K_t)'*A$. It is easy to check that $(K_t)'=- t^{-2}\,\psi_t$,
and then the first inequality claimed in the lemma follows just writing
$$\bigl|K_{t} * A'\,(x)\bigr|^p\,t^{p-\alpha p} = \bigl|t^{\frac1p-1-\alpha}\,\psi_t * A\,(x)\bigr|^p \,\frac1t.$$

To prove \rf{eqdr521}, 
first we calculate the Fourier transform of $\psi_t$. Notice that 
$K_t= c\,(P_t)'$, where $P_t$ is the
Poisson kernel and $c$ is some absolute constant. So,
\begin{equation}\label{eqpsi9}
\wh{\psi_t}(\xi) = c\,t^2\xi^2 \,e^{-2\pi\,t\,|\xi|}.
\end{equation}

Consider a radial $\CC^\infty$ function $\eta$ whose Fourier transform is supported in the annulus
$A(0,1/2,3/2)$, and setting $\eta_{(k)}(x)=\eta_{2^{-k}}(x) = 2^k\,\eta(2^k\,x)$,
\begin{equation}\label{eqnu3}
\sum_{k\in\Z} \wh{\eta_{(k)}}(\xi)=1\qquad \mbox{for all $\xi\neq0$.}
\end{equation}
Then  we have
$$\|A\|_{\dot{B}_{p,p}^{1+\alpha-1/p}}^p\approx \sum_{k\in\Z}\|2^{k(1+\alpha-\frac1p)}\eta_{(k)}*A\|_p^p.$$

Notice now that there  exists some Schwartz function $\tau$ such that 
$\eta = \psi * \tau$. Indeed, we only have to take
$$\wh \tau(\xi) = \frac{\wh \eta(\xi)}{c\,\xi^2 \,e^{-2\pi\,|\xi|}},$$
so that $\wh \tau\in\CC_c^\infty$. Similarly, for any $s\in[1,2]$, we take 
some Schwartz function $\tau^s$ such that 
$$\eta = \psi_s * \tau^s.$$
Then, for every $k\in\Z$ and every $s\in[1,2]$, we have
$$\eta_{(k)} *A = \psi_{s2^{-k}} * \tau^s_{2^{-k}} *A,$$
where $\tau^s_{2^{-k}}(x) = 2^k\tau^s(2^kx)$. Thus,
$$\|\eta_{(k)} *A \|_p \leq \|\tau^s_{2^{-k}}\|_1\,\|\psi_{s2^{-k}} * A\|_p
\leq c\,\|\psi_{s2^{-k}} * A\|_p,$$
where we took into account that $\|\tau^s_{2^{-k}}\|_1=\|\tau^s\|_1\leq c$ for some 
constant independent of $s$. Then we deduce
$$\|A\|_{\dot{B}_{p,p}^{1+\alpha-1/p}}^p\lesssim \sum_{k\in\Z}\|2^{k(1+\alpha-\frac1p)}\psi_{s2^{-k}}*A\|_p^p
\qquad\mbox{for all $s\in[1,2]$.}$$
As a consequence, by Fubini and a change of variables,
\begin{align*}
\|A\|_{\dot{B}_{p,p}^{1+\alpha-1/p}}^p&\lesssim  \sum_{k\in\Z}\int_1^2\|2^{k(1+\alpha-\frac1p)}\psi_{s2^{-k}}*A\|_p^p
\,\frac{ds}s \\& = \sum_{k\in\Z}
\int_{2^{-k}}^{2^{-k+1}}\|2^{k(1+\alpha-\frac1p)}\psi_{t}*A\|_p^p\frac{dt}t
\approx \int_0^\infty 
\|t^{\frac1p-1-\alpha}\,\psi_t * A\|_p^p \,\frac{dt}t.
\end{align*}
This proves the lemma.
\end{proof}

\begin{lemma}\label{lem6}
Under the assumptions and notation above,
$$\sum_{Q\in\WW(\Omega)}\int_{3Q}I_2(w,Q)^p \ell(Q)^{p-\alpha p}\,dm(w) \lesssim 
\int_0^\infty 
\|t^{\frac1p-1-\alpha}\,\psi_t * A\|_p^p \,\frac{dt}t,$$
where $\psi_t$ is as in Lemma \ref{lem5}.
\end{lemma}

\begin{proof}
The arguments are quite similar to the ones for Lemma \ref{lem5}.
Consider $Q\in\WW(\Omega)$, and set $Q=(a_1,a_2]\times (b_1,b_2]$.
Recall that, for $w=a+ib\in3Q$,
$$I_2(w,Q)=\biggl|\int_{\R} 
\frac{\left(\frac s{1+\theta^2}\right)^2-\left(x-a-\frac{\theta s}{1+\theta^2}\right)^2}{\bigl[\bigl(x-a-\frac{\theta s}{1+\theta^2}\bigr)^2 + \left(\frac s{1+\theta^2}\right)^2\bigr]^2}
\,A'(x)
\,dx\biggr|,$$
where $s\equiv s_Q(w) = b - g_Q(a),$ and 
$\theta\equiv\theta_Q$ is the slope of the affine line defined by $g_Q$.
For $t>0$, consider the kernel 
$$J_t(x) = \frac{t^2-x^2}{[x^2 + t^2]^2}.$$
Now we have
$$I_2(w,Q)= \Bigl|J_{\frac s{1+\theta^2}} * A'\,\Bigl(a+\frac{\theta s}{1+\theta^2}\Bigr)\Bigr|,$$
and so
$$\int_{3Q}I_2(w,Q)^p \,dm(w) = \int_{a_1-\ell(Q)}^{a_2+\ell(Q)} \int_{t_1-\ell(Q)}^{t_2+\ell(Q)}
\Bigl|J_{\frac s{1+\theta^2}} * A'\,\Bigl(x+\frac{\theta s}{1+\theta^2}\Bigr)\Bigr|^p \,dt\,dx.$$
Then, assuming $\delta$ small enough, for each $x\in p_1(3Q)$ we have
\begin{align*}
\int_{t_1-\ell(Q)}^{t_2+\ell(Q)}
\Bigl|J_{\frac s{1+\theta^2}} * A'\,\Bigl(x+\frac{\theta s}{1+\theta^2}\Bigr)\Bigr|^p \,dt 
&= 
\int_{t_1-\ell(Q)}^{t_2+\ell(Q)}
\Bigl|J_{\frac{t-g_Q(x)}{1+\theta^2}} * A'\,\Bigl(x+\frac{\theta s}{1+\theta^2}\Bigr)\Bigr|^p \,dt 
\\
&\leq
\int_{t_1-\frac32\ell(Q)}^{t_2+\frac32\ell(Q)}
\Bigl|J_{t-A(a_1)} * A'\,\Bigl(x+\frac{\theta s}{1+\theta^2}\Bigr)\Bigr|^p \,dt 
.
\end{align*}
Also, it follows easily  that 
\begin{align*}
\int_{a_1-\ell(Q)}^{a_2+\ell(Q)} 
\int_{t_1-\frac32\ell(Q)}^{t_2+\frac32\ell(Q)}&
\Bigl|J_{t-A(a_1)} * A'\,\Bigl(x+\frac{\theta s}{1+\theta^2}\Bigr)\Bigr|^p \,dt\,dx\\
&\leq 
\int_{a_1-2\ell(Q)}^{a_2+2\ell(Q)} 
\int_{t_1-\frac32\ell(Q)}^{t_2+\frac32\ell(Q)}
\bigl|J_{t-A(a_1)} * A'\,(x)\bigr|^p \,dt\,dx\\
& \leq
\int_{a_1-2\ell(Q)}^{a_2+2\ell(Q)} 
\int_{t_1-2\ell(Q)}^{t_2+2\ell(Q)}
\bigl|J_{t-A(x)} * A'\,(x)\bigr|^p \,dt\,dx.
\end{align*}
As a consequence,
\begin{align}\label{eqdr5'}
\sum_{Q\in\WW(\Omega)}\int_{3Q}I_2(w,Q)^p \ell(Q)^{p-\alpha p}\,dm(w)
&\leq\sum_{Q\in\WW(\Omega)} \iint_{(x,t)\in 5Q} 
\bigl|J_{t-A(x)} * A'\,(x)\bigr|^p \,t^{p-\alpha p}\,dt\,dx \\
& \approx \iint_{(x,t)\in \Omega} 
\bigl|J_{t-A(x)} * A'\,(x)\bigr|^p \,t^{p-\alpha p}\,dt\,dx \nonumber\\
& = \int_{x\in\R}\int_0^\infty 
\bigl|J_{t} * A'\,(x)\bigr|^p \,t^{p-\alpha p}\,dt\,dx.\nonumber
\end{align}

We denote $\vphi_t =t^2 (J_t)'$, so that
$J_{t} * A' = (J_t)'*A= t^{-2}\,\vphi_t*A$. Moreover, it turns out that
$\vphi_t(x)=t^{-1}\vphi(t^{-1}x)$, where $\vphi\equiv\vphi_1$. Then we have
\begin{equation}\label{eqdr6'}
\int_{x\in\R}\int_0^\infty 
\bigl|J_{t} * A'\,(x)\bigr|^p \,t^{p-\alpha p}\,dt\,dx = \int_0^\infty 
\|t^{\frac1p-1-\alpha}\,\vphi_t * A\|_p^p \,\frac{dt}t.
\end{equation}

To calculate the Fourier transform of $\vphi$, notice that 
$J_t(x)= c\,(Q_t)'(x)$, where $Q_t$ is the
conjugated Poisson kernel and $c$ is some absolute constant. So,
$$\wh{J_t}(\xi) = c\,\xi\,\operatorname{sgn}(\xi) e^{-2\pi t|\xi|} = 
c\,|\xi|\, e^{-2\pi t|\xi|},$$
and
$$\wh{\vphi_t}(\xi) = c\,t^2\xi\,|\xi|\, e^{-2\pi \,t|\xi|}.$$
So, recalling \rf{eqpsi9}, it turns out that $\wh{\vphi_t}(\xi) =c\,\operatorname{sgn}(\xi)\,\wh{\psi_t}(\xi).$ That is, $\vphi_t$ is the Hilbert transform of $\psi_t$, modulo a constant factor.
Thus, denoting by $H$ the Hilbert transform,
$$\|\vphi_t*A\|_p =c\,\|H(\psi_t*A)\|_p\leq c\,\|\psi_t*A\|_p,$$
and so, by \rf{eqdr6'},
$$\int_{x\in\R}\int_0^\infty 
\bigl|J_{t} * A'\,(x)\bigr|^p \,t^{p-\alpha p}\,dt\,dx \leq c\int_0^\infty 
\|t^{\frac1p-1-\alpha}\,\psi_t * A\|_p^p \,\frac{dt}t,$$
and the lemma follows.
\end{proof}

\begin{proof}[\bf Proof of the Main Lemma \ref{mainlemma}]
By \rf{eqde59} and Lemmas \ref{lem5} and \ref{lem6}, assuming $\delta$ small enough, we get
\begin{multline*}
\sum_{Q\in\WW(\Omega)}\int_{3Q}\left|\int_{\R} \frac{A'(x)}{(x+i\,g_Q(x)-w)^2} \,dx\right|^p
\ell(Q)^{p-\alpha p}\,dm(w)\\
 \gtrsim  (1-c\,\delta^p)
\int_0^\infty 
\|t^{\frac1p-1-\alpha}\,\psi_t * A\|_p^p \,\frac{dt}t
 \gtrsim \|A\|_{\dot{B}_{p,p}^{1+\alpha-1/p}}^p.
\end{multline*}
Together with \rf{eqde56}, this implies that
$$
\int_\Omega|\partial B\chi_\Omega(z)|^p\,\dist(z,
\partial\Omega)^{p-\alpha p}dm(z) \gtrsim 
\|A\|_{\dot{B}_{p,p}^{1+\alpha-1/p}}^p -c\,\delta^p\|A\|_{\dot{B}_{p,p}^{1+\alpha-1/p}}^p\gtrsim
\|A\|_{\dot{B}_{p,p}^{1+\alpha-1/p}}^p,
$$
for $\delta$ small enough again.
\end{proof}


\vvv
\section{The proof of Theorem \ref{teodom}}\label{sec5}

Let $\Omega\subset\C$ be a bounded domain which is $(\delta,R)$-Lipschitz.
We have to show that
\begin{equation}\label{eqeq89'}
\|N\|_{\dot{B}_{p,p}^{\alpha-1/p}(\partial\Omega)}^p\lesssim\|B(\chi_\Omega)\|_{\dot W^{\alpha,p}(\Omega)}^p+ \HH^1(\partial\Omega)^{2-\alpha p}.
\end{equation}
To this end we will prove:

\begin{lemma}\label{mainlemma2}
Let $\Omega\subset\C$ be a $(\delta,R)$-Lipschitz domain. Let $1<p<\infty$ and $0<\alpha\leq1$ be such
that $\alpha p>1$. If $\delta$ is small enough, then
\begin{equation}\label{eqmain73'}
\|N\|_{\dot B_{p,p}^{\alpha-1/p}(\partial\Omega)}^p\lesssim \int_\Omega|\partial B\chi_\Omega(z)|^p\,\dist(z,
\partial\Omega)^{p-\alpha p}\,dm(z) + \HH^1(\partial\Omega)^{2-\alpha p}.
\end{equation}
\end{lemma}

Let us show first that this result yields Theorem \ref{teodom} as
an easy consequence.

\begin{proof}[\bf Proof of Theorem \ref{teodom}]
In the case $\alpha=1$, it is clear that \rf{eqeq89'} follows from \rf{eqmain73'}.
So assume that $0<\alpha<1$. In this case, if $\theta$ is chosen small enough,
from Lemma \ref{mainlemma2}, we have
$$\theta^{p-\alpha p}\biggl(\HH^1(\partial\Omega)^{2-\alpha p}+\! \int_\Omega|\partial B\chi_\Omega(z)|^p\,\dist(z,
\partial\Omega)^{p-\alpha p}dm(z)\biggr) \geq 2 c_3\,\theta^{2p-\alpha p}
\|N\|_{B_{p,p}^{\alpha-1/p}(\partial\Omega)}^p,$$
where $c_3$ appears in \rf{eqdf890}. Then, by Lemma \ref{lemkeyalfa},
\begin{equation}\label{eqcor5'}
\theta^{p-\alpha p}\,\HH^1(\partial\Omega)^{2-\alpha p} +\|B(\chi_\Omega)\|_{\dot W^{\alpha,p}(\Omega)}^p\gtrsim \theta^{p-\alpha p} \!\!\int_\Omega|\partial B\chi_\Omega(z)|^p\,\dist(z,
\partial\Omega)^{p-\alpha p}\,dm(z).
\end{equation}
Together with Lemma \ref{mainlemma2} again this gives \rf{eqeq89'}.
\end{proof}

\vv

Suppose that $\Omega$ is simply connected.
Consider and arc length   parameterization of $\partial\Omega$ given by
$\gamma:S^1(0,r_0)\to\partial \Omega$, where $2\pi r_0=\HH^1(\partial \Omega)$.
Recall that, for a function $f:S^1(0,r_0)\to\R$ and $0<\alpha<1$ and $1<p<\infty$,
$$\|f\|_{\dot{B}_{p,p}^{\alpha}(S^1(0,r_0))}^p = \iint_{(s,t)\in S^1(0,r_0)\times S^1(0,r_0)}
\frac{|f(s)- f(t)|^p}{|s-t|^{\alpha p+1}}\,ds\,dt.$$
Then, taking into account that $\|N\circ\gamma\|_{\dot{B}_{p,p}^{\alpha-1/p}(S^1(0,r_0))}
\approx \|N\|_{\dot{B}_{p,p}^{\alpha-1/p}(\partial\Omega)}$, \rf{eqmain73'} is equivalent to
\begin{align}\label{eqd477}
\iint_{(s,t)\in S^1(0,r_0)\times S^1(0,r_0)}\!\! &
\frac{|N(\gamma(s))- N(\gamma(t))|^p}{|s-t|^{\alpha p}}\,ds\,dt\\&
\lesssim \int_\Omega|\partial B\chi_\Omega(z)|^p\,\dist(z,
\partial\Omega)^{p-\alpha p}\,dm(z) + \HH^1(\partial\Omega)^{2-\alpha p}.\nonumber
\end{align}

We will use the following notation: given $a>1$ and a small arc $I\subset S^1(0,r_0)$, 
we denote by $aI$ the arc of $S^1(0,r_0)$ with the same mid point as $I$ and length $\ell(aI)
=a\,\ell(I)$.

The main step for the proof of Lemma \ref{mainlemma2} consists of next lemma.

\begin{lemma}\label{lempart}
Suppose that $\Omega$ is simply connected.
Under the assumptions and notation above,
consider an arc $I\subset S^1(0,r_0)$ with $\ell(I)\leq R/4$ and denote by
$s_1,s_2$ the end points of $2I$. Let $a=\delta^{1/2}\,\ell(I)$. Then, for $0<\delta\ll1$ small
enough we have
\begin{align}\label{eqds89}
\iint_{\begin{subarray}{l} s\in I\\ t\in 1.1I
 \end{subarray}} &\frac{|N(\gamma(s))- N(\gamma(t))|^p}{|s-t|^{\alpha p}}\,ds\,dt\lesssim \int_{\Omega\cap B(z_I,4\ell(I))} |\partial B\chi_{\Omega}(z)|^p\,
 \dist(z,\partial\Omega)^{p-\alpha p}\,dm(z)\\
&\quad\quad\quad\quad+ \sum_{i=1}^2
\int_{|s-s_i|\leq 7a}\int_{t\in S^1(0,r_0)} \!\!\!\frac{|N(\gamma(s))-N(\gamma(t))|^{p}}{|s-t
|^{\alpha p}}\,dt\,ds
+c(\delta)\,\ell(I)^{2-\alpha p}.\nonumber
\end{align}
\end{lemma}

\begin{proof}
Denote by $s_I$ the mid point of $I$ and set $z_I=\gamma(s_I)$.
Let $A:\R\to\R$ be the Lipschitz function whose graph $\Gamma$ coincides with $\partial \Omega$ on $B(z_I,R)$, so that $\Omega\cap B(z_I,R)$ lies above $\Gamma$, after
a suitable rotation. Notice that $\gamma(6I)\subset B(z_I,R)$, since 
$\ell(I)=\HH^1(\gamma(I))\leq R/4$. 
Let $J\subset\R$ be the interval such that $\{(x,A(x)):x\in J\}=\gamma(I)$, so that
$\{(x,A(x)):x\in 5J\}\subset \gamma(6I)\subset B(z_I,R)$.
Observe that $$\gamma(1.1I)\subset
\{(x,A(x)):x\in 1.1J\}.$$ 
It also immediate to check that $\ell(I)\approx\ell(J)$.
Moreover, translating $\Gamma$  slightly if necessary, we may assume that
 one of the endpoints 
of $\gamma(I)$ lies on the horizontal coordinate axis. Notice that then, since $\|A'\|_\infty\leq\delta$, by the mean value theorem, it turns out that 
\begin{equation}\label{eqrec54}
|A(x)|\lesssim\delta\,\ell(J)\qquad \mbox{for all $x\in5J$.}
\end{equation}
Moreover, we will assume that $A$ is defined in the whole of $\R$ and that $\|A\|_\infty\lesssim
\delta\ell(J)$ and $\|A'\|_\infty\lesssim \delta$.

Denote $z_i=\gamma(s_i)$, for $i=1,2$ (recall that $s_1,s_2$ are the end points of $2I$).
Also, let $x_i$ be such that $z_i=(x_i,A(x_i))$. Assume that $x_1<x_2$. 
Let $\vphi:\R\to\R$ be a $\CC^\infty$ function which equals $1$ on $[x_1,x_2]$ and vanishes on $\R\setminus [x_1-a,\,x_2+a]$, with $\|\vphi'\|_\infty\lesssim 1/a$ (recall that $a=\delta^{1/2}\ell(I)
\approx \delta^{1/2}\ell(J)$). 
Observe that, since we are assuming $\delta$ to be very small,
$$[x_1-a,x_2+a]\subset 3J$$

We consider the auxiliary Lipschitz function $\wt A=\vphi\,A$ and
its graph $\wt \Gamma$.
Let $$\wt\Omega=\{(x,y)\in\C:\,y>\wt A(x)\},$$
and denote by $\wt N(x)$ the outward unit normal at $(x,\wt A(x))\in\wt \Gamma$.
By Corollary \ref{coro40}, we have
$$\|\wt N\|_{\dot{B}_{p,p}^{\alpha-1/p}}\approx
\int_{\wt \Omega}|\partial B\chi_{\wt \Omega}(z)|^p\,\dist(z,
\partial\wt \Omega)^{p-\alpha p}\,dm(z).
$$
Indeed, using \rf{eqrec54},
$$\|\wt A'\|_\infty\leq \|A'\|_\infty + \|\vphi'\|_\infty\,\|\chi_{3J}A\|_\infty\leq
\delta + \frac{c}{a}\,\delta \,\ell(J)\lesssim \delta + \delta^{1/2} \approx\delta^{1/2},$$
and thus the assumption on the small slope of the Lipschitz function in Theorem \ref{teopri} holds for $\delta$ small 
enough.

On the other hand, since $\wt N$ coincides with $N$ on $\gamma(1.1I)$, 
$$\|\wt N\|_{\dot{B}_{p,p}^{\alpha-1/p}}^p\approx 
 \iint_{(x,y)\in\R^2} \!\!\frac{|\wt N(x)- \wt N(y)|^p}{|x-y|^{\alpha p}}\,dx\,dy \gtrsim
  \iint_{\begin{subarray}{l} s\in I\\ t\in 1.1I
 \end{subarray}}\!\! \frac{|N(\gamma(s))- N(\gamma(t))|^p}{|s-t|^{\alpha p}}\,ds\,dt
.$$
Therefore, to prove the lemma it is enough to show that $
\int_{\wt \Omega}|\partial B\chi_{\wt \Omega}|^p\dist(\cdot,
\partial\wt \Omega)^{p-\alpha p}dm$ is bounded above by
the right side of \rf{eqds89}.

Consider the rectangle
$$V=[x_1-2a,\,x_2+2a]\times [-a,a].$$
To estimate $\|\partial B(\chi_{\wt \Omega})\|_{L^p(\wt\Omega)}^p$, we write
\begin{align}\label{eqal80}
\int_{\wt \Omega}|\partial B\chi_{\wt \Omega}|^p\,\dist(\cdot,
\partial\wt \Omega)^{p-\alpha p}\,dm& 
\leq 
\int_{\wt\Omega\setminus V}|\partial B\chi_{\wt \Omega}|^p\,\dist(\cdot,
\partial\wt \Omega)^{p-\alpha p}\,dm
\\
&\quad\!\!+ \int_{ \wt\Omega\cap V\setminus (B(x_1,4a)\cup B(x_2,4a))} \!\!|\partial B(\chi_{\wt \Omega})|^p\dist(\cdot,
\partial\wt \Omega)^{p-\alpha p}\,dm
\nonumber\\
&\quad\!\!
+  \int_{\wt\Omega\cap (B(x_1,4a)\cup B(x_2,4a))}|\partial B(\chi_{\wt \Omega})|^p\dist(\cdot,
\partial\wt \Omega)^{p-\alpha p}\,dm.\nonumber
\end{align}
Let us deal with the first integral on the right side. 
To this end, consider the upper half plane $\Pi=\{(x,y)\in\C:y>0\}$. 
Recall that $\partial B\chi_\Pi(z)=0$ for $z\in\Pi$. Therefore, for $z\in\wt \Omega\setminus V\subset
\Pi$, using the first identity in \rf{eqtyu},
\begin{align*}
|\partial B(\chi_{\wt \Omega})(z)| & = |\partial B(\chi_{\wt \Omega})(z)- \partial B(\chi_{\Pi})(z)|
\\
&\leq 
\int_{\Pi\Delta\wt \Omega}\frac1{|z-w|^3}\,dm(w) 
\leq \frac{m(\Pi\Delta\wt \Omega)}{\dist(z,\Pi\Delta\wt \Omega)^3}\lesssim\frac{\delta \,
\ell(J)^2}{\dist(z,\Pi\Delta\wt \Omega)^3} .
\end{align*}
It is easy to check that if $z\not\in V$, then $\dist(z,\Pi\Delta\wt \Omega)\gtrsim a$.
Using also the fact that $\dist(z,\Pi\Delta\wt \Omega)\approx |z-z_I|$ for $|z-z_I|\geq 4\ell(J)$,
we obtain
\begin{align}\label{eqs31}
\int_{\wt\Omega\setminus V}|\partial B(\chi_{\wt \Omega})|^p\dist(\cdot,
\partial\wt \Omega)^{p-\alpha p}\,dm
&\lesssim \int_{B(z_I,4\ell(J))} \!\!
\frac{\delta^p \,
\ell(J)^{2p}}{a^{3p}}\,\ell(J)^{p-\alpha p}
\,dm(z) \\
&\quad + 
\int_{\C\setminus B(z_I,4\ell(J))} 
\frac{\delta^p \,
\ell(J)^{2p}}{|z-z_I|^{3p}}\,|z-z_I|^{p-\alpha p}\nonumber
\,dm(z)\\
&\lesssim \frac{\delta^p \,
\ell(J)^{3p+2-\alpha p}}{a^{3p}} + \frac{\delta^p \,
\ell(J)^{2p}}{\ell(J)^{2p+\alpha p -2}} \approx \delta^{-p/2} \,
\ell(I)^{2-\alpha p}.\nonumber
\end{align}

Let us turn our attention to the second integral on the right side of \rf{eqal80} now.
In this case, using that $$\wt\Omega\cap V\setminus (B(x_1,4a)\cup B(x_2,4a)) 
= \Omega\cap V\setminus (B(x_1,4a)\cup B(x_2,4a)),$$ 
we write
\begin{align*}
&\int_{\wt\Omega\cap V\setminus (B(x_1,4a)\cup B(x_2,4a))}|\partial B(\chi_{\wt \Omega})|^p
\dist(\cdot,
\partial\wt \Omega)^{p-\alpha p}\,dm\\
&\qquad\qquad\qquad\lesssim \int_{\Omega\cap B(z_I,4\ell(I))} |\partial B(\chi_{\Omega})|^p\dist(\cdot,
\partial\wt \Omega)^{p-\alpha p}\,dm\\
&\qquad\qquad\qquad\quad+\int_{V\setminus (B(x_1,4a)\cup B(x_2,4a))} |\partial B(\chi_{\Omega})-\partial B(\chi_{\wt \Omega})|^p\dist(\cdot,
\partial\wt \Omega)^{p-\alpha p}\,dm.
\end{align*}
We estimate the last integral arguing as above. Observe that for $z\in V\setminus (B(x_1,4a)\cup B(x_2,4a))$, we have $\dist(z,\Omega\Delta\wt\Omega)\gtrsim a$. Then
we have
\begin{align*}
|\partial B(\chi_{\Omega})(z)-\partial B(\chi_{\wt \Omega})(z)|
&\leq \int_{\Omega\Delta\wt\Omega} \frac1{|z-w|^3}\,dm(w)\\
& \leq \int_{|z-w|\geq c^{-1}a} \frac1{|z-w|^3}\,dm(w)\lesssim \frac1a.
\end{align*}
As a consequence, since $\dist(z,\partial\wt\Omega)\lesssim a$ for $z\in V$,
\begin{align*}
\int_{V\setminus (B(x_1,4a)\cup B(x_2,4a))} &|\partial B(\chi_{\Omega})-\partial B(\chi_{\wt \Omega})|^p\,\dist(\cdot,
\partial\wt \Omega)^{p-\alpha p}\,dm\\
&\lesssim \int_{V\setminus (B(x_1,4a)\cup B(x_2,4a))} \frac{a^{p-\alpha p}}{a^p}\,dm
 \lesssim \frac{a\,\ell(J)}{a^{\alpha p}} \approx \delta^{(1-\alpha p)/2} \,\ell(I)^{2-\alpha p},\nonumber
\end{align*}
and thus, taking into account also that $\dist(z,\partial\wt\Omega)= \dist(z,\partial\Omega)$
in the domain of integration,
\begin{align}\label{eqs32}
&\int_{V\setminus (B(x_1,4a)\cup B(x_2,4a))} |\partial B(\chi_{\wt \Omega})|^p \dist(\cdot,\partial\wt\Omega)^{p-\alpha p}\,dm \\
&\qquad\lesssim \nonumber
\int_{V\setminus (B(x_1,4a)\cup B(x_2,4a))} |\partial B(\chi_{\Omega})|^p 
\dist(\cdot,\partial\Omega)^{p-\alpha p}\,dm + \delta^{(1-\alpha p)/2} \,\ell(I)^{2-\alpha p}\\
&\qquad \leq
\int_{\Omega\cap B(z_I,4\ell(I))} |\partial B(\chi_{\Omega})|^p \dist(\cdot,\partial\Omega)^{p-\alpha p}\,dm+\delta^{(1-\alpha p)/2} \,\ell(I)^{2-\alpha p}.\nonumber
\end{align}

It remains to estimate the last integral in \rf{eqal80}. First we deal with the integral
on $B(x_1,4a)$.
Let $\psi:\R\to\R$ be a $\CC^\infty$ function such that $0\leq\psi\leq1$ which equals $1$  
in $[x_1-5a,x_1+5a]$ and vanishes in $\R\setminus[x_1-6a,x_1+6a]$, with $\|\psi'\|_\infty\leq c/a$.
Denote $A_0 = \psi\,\wt A = \psi\,\vphi\,A$ and set 
$$\Omega_0=\{(x,y)\in\C:\,y>A_0(x)\}.$$
Then we have
\begin{align}\label{eqed5}
\int_{\wt\Omega\cap B(x_1,4a)} |\partial B(\chi_{\wt \Omega})|^p\,\dist(\cdot,
\partial& \wt \Omega  )^{p-\alpha p}\,dm
\lesssim \int_{\wt\Omega\cap B(x_1,4a)} |\partial B(\chi_{\Omega_0})|^p\dist(\cdot,
\partial\wt \Omega)^{p-\alpha p}\,dm\\
&+\int_{\wt\Omega\cap B(x_1,4a)} |\partial B(\chi_{\Omega_0})-\partial B(\chi_{\wt \Omega})|^p\dist(\cdot,
\partial\wt \Omega)^{p-\alpha p}\,dm.\nonumber
\end{align}
Since $A_0$ coincides with $\wt A$ in $[x_1-5a,x_1+5a]$, 
 it turns out that, for $z\in B(x_1,4a)$, $\dist(z,\Omega_0\Delta\wt\Omega)\gtrsim a$, and thus
$$|\partial B(\chi_{\Omega_0})(z)-\partial B(\chi_{\wt \Omega})(z)|
\leq \int_{\Omega_0\Delta\wt\Omega} \frac1{|z-w|^3}\,dm(w)
\leq \int_{|z-w|\gtrsim a} \frac1{|z-w|^3}\,dm(w)\lesssim\frac1a
.$$
Therefore, using that $\dist(\cdot,\partial\wt\Omega)\lesssim a$ on $B(x_1,4a)$,
\begin{align}\label{eqdif63}
\int_{B(x_1,4a)} |\partial B(\chi_{\Omega_0})-\partial B(\chi_{\wt \Omega})|^p
\dist(\cdot,\partial\wt\Omega)^{p-\alpha p}\,dm &
\lesssim \frac{a^{p-\alpha p} \,a^2}{a^p}\approx \delta^{(2-\alpha p)/2}\,\ell(I)^{2-\alpha p}.
\end{align}

For the first integral on the right side of \rf{eqed5} we use the fact that $\wt\Omega\cap B(x_1,4a)=\Omega_0\cap B(x_1,4a)$, and moreover that $\dist(\cdot,\partial\wt\Omega)=\dist(\cdot,\partial\Omega_0)$ in $\wt\Omega\cap B(x_1,4a)$. Then, by Corollary \ref{coro40} applied to
$\Omega_0$, we get
\begin{align}\label{eqd41}
\int_{\wt\Omega\cap B(x_1,4a)} |\partial B(\chi_{\Omega_0})|^p\dist(\cdot,\partial\wt\Omega)^{p-\alpha p}\,dm& \leq \int_{\Omega_0} |\partial B(\chi_{\Omega_0})|^p\dist(\cdot,\partial\Omega_0)^{p-\alpha p}\,dm\\ &\lesssim
 \|A_0\|_{\dot{B}_{p,p}^{1+\alpha-1/p}}^p.\nonumber
\end{align}
We have  
\begin{equation}\label{eqd42}
\|A_0\|_{\dot{B}_{p,p}^{1+\alpha-1/p}}\approx \|A_0'\|_{\dot{B}_{p,p}^{\alpha-1/p}}
\leq \|\vphi\psi A'\|_{\dot{B}_{p,p}^{\alpha-1/p}} + \|(\vphi\psi)' A\|_{\dot{B}_{p,p}^{\alpha-1/p}}.
\end{equation}
From Lemma \ref{lemaux1}, we deduce that
\begin{align*}
\|\vphi\psi A'\|_{\dot{B}_{p,p}^{\alpha-1/p}}^p&\lesssim \iint |\vphi\psi \Delta_h A'(x)|^p\,\frac{dh}{h^{\alpha p}}\,dx+
 \|\vphi\psi\|_{\dot B_{p,p}^{\alpha-1/p}}^p\|A'\|_\infty^p.
\end{align*}

It is easy to check that
$\|\vphi\psi\|_{B_{p,p}^{1-1/p}}^p\lesssim a^{2-\alpha p}$. Indeed, this a straightforward consequence of
Lemma \ref{lemaux2}: since $\vphi\psi$ is supported on an interval with length $\lesssim a$ and 
$\|(\vphi\psi)'\|_\infty\lesssim1/a$,
$$\|\vphi\psi\|_{\dot{B}_{p,p}^{\alpha-1/p}}\lesssim a^{1-\alpha+2/p}\,\|(\vphi\psi)'\|_\infty
\lesssim a^{-\alpha+2/p},
$$
and so our last claim holds.
Then we obtain
\begin{align}\label{eqd43}
\|\vphi\psi A'\|_{\dot{B}_{p,p}^{\alpha-1/p}}^p&\lesssim \iint |\chi_{[x_1-6a,x_1+6a]}\, \Delta_h A'(x)|^p\,\frac{dh}{h^{\alpha p}}\,dx
+  \delta^p\,a^{2-\alpha p}.
\end{align}
We split the integral on the right side above as follows:
$$\int_{|x-x_1|\leq6a}\int_{|h|\leq \ell(I)/2}\!\!| \Delta_h A'(x)|^p\,\frac{dh}{h^{\alpha p}}\,dx
+ \int_{|x-x_1|\leq6a}\int_{|h|> \ell(I)/2}\!\!|\Delta_h A'(x)|^p\,\frac{dh}{h^{\alpha p}}\,dx=I_1+I_2.$$
For $I_2$ we use the rough estimate $|\Delta_h A'(x)|\leq 2\delta$, and then, using that
$\alpha p>1$, we get
$$I_2\leq 2^p\delta^p\int_{|x-x_1|\leq6a}\int_{|h|> \ell(I)/2}\frac{dh}{h^{\alpha p}}\,dx
\lesssim \delta^p a\,\ell(I)^{1-\alpha p} =c(\delta) \ell(I)^{2-\alpha p}.
$$

For $I_1$, recall that  
$|A'(x)-A'(y)|\approx|N(x,A(x))-N(y,A(y))|$ for $x,y\in 5J$, by \rf{eqn62}. Thus we get
\begin{align*}
I_1 &= \int_{|x-x_1|\leq6a}\int_{|x-y|\leq \ell(I)/2} \frac{|A'(x)-A'(y)|^p}{|x-y|^{\alpha p}}\,dy\,dx\\
&\lesssim \int_{|s-s_1|\leq 7a}\int_{t\in S^1(0,r_0)} \frac{|N(\gamma(s))-N(\gamma(t))|^p}{|s-t|^{\alpha
p}}\,dt\,ds.
\end{align*}
Therefore, from \rf{eqd43} we derive
\begin{equation}\label{eqd43.5}
\|\vphi\psi A'\|_{\dot{B}_{p,p}^{\alpha-1/p}}^p\lesssim 
\int_{|s-s_1|\leq 7a}\int_{t\in S^1(0,r_0)} \frac{|N(\gamma(s))-N(\gamma(t))|^p}{|s-t|^{\alpha p}}\,dt\,ds
+  c(\delta)\,\ell(I)^{2-\alpha p}.
\end{equation}

To estimate $\|(\vphi\psi)' A\|_{\dot{B}_{p,p}^{\alpha-1/p}}$ we use that $(\vphi\psi)' A$ is a 
Lipschitz function supported on $[x_1-6a,x_1+6a]$ satisfying
$$\|[(\vphi\psi)' A]'\|_\infty \leq 
\|(\vphi\psi)'' A\|_\infty + \|(\vphi\psi)' A'\|_\infty\lesssim \frac{\delta\,\ell(I)}{a^2}
+ \frac\delta a \lesssim \frac1{\ell(I)}.$$
Then Lemma \ref{lemaux2} tells us that
\begin{equation}\label{eqd44}
\|(\vphi\psi)' A\|_{\dot{B}_{p,p}^{\alpha-1/p}}\lesssim a^{1-\alpha +2/p}\,\|[(\vphi\psi)' A]'\|_\infty
\lesssim c(\delta)\,\ell(I)^{-\alpha+2/p}.
\end{equation}
From \rf{eqed5}, \rf{eqdif63},  \rf{eqd41}, \rf{eqd42}, \rf{eqd43.5},
and \rf{eqd44}, we get
\begin{align}\label{eqs33}
\int_{\wt\Omega\cap B(x_1,4a)} |\partial B( & \chi_{\wt\Omega})|^p \dist(\cdot,\partial\wt\Omega)\,dm\\
&\lesssim
\int_{|s-s_1|\leq 7a}\int_{t\in S^1(0,r_0)} \!\!\!\!\!\frac{|N(\gamma(s))-N(\gamma(t))|^p}{
|s-t|^{\alpha p}}\,dt\,ds
+ c(\delta) \ell(I)^{2-\alpha p}.   \nonumber
\end{align}
An analogous inequality holds the integral over $B(x_2,4a)$.

Plugging the estimates obtained in \rf{eqs31}, \rf{eqs32} and \rf{eqs33} into \rf{eqal80}, we get
\begin{align*}
\|\partial B(\chi_{\wt \Omega})\,&\dist(\cdot,\partial\wt\Omega)^{1-\alpha}\|_{L^p(\wt\Omega)}^p\lesssim \int_{\Omega\cap B(z_I,4\ell(I))} |\partial B(\chi_{\Omega})|^p\dist(\cdot,\partial
\Omega)^{p-\alpha p}\,dm\\
&\quad+ \sum_{i=1}^2\int_{|s-s_i|\leq 7a}\int_{t\in S^1(0,r_0)} \!\!\!\frac{|N(\gamma(s))-N(\gamma(t))|^p}{|s-t|^{\alpha p}}\,dt\,ds+c(\delta)\,\ell(I)^{2-\alpha p},
\end{align*}
which proves the lemma.
\end{proof}

\vspace{5mm}

\begin{proof}[\bf Proof of Lemma \ref{mainlemma2}]
Suppose first that $\Omega$ is simply connected, and let $\gamma:S^1(0,r_0)\to\partial\Omega$
be an arc length parameterization of $\partial\Omega$. We will prove \rf{eqd477}.

Let $I\subset S^1(0,r_0)$ be an arc $\ell(I)\leq R/4$. We set
\begin{align*}
\iint_{\begin{subarray}{l} s\in I\\ t\in S^1(0,r_0)
 \end{subarray}} \frac{|N(\gamma(s))- N(\gamma(t))|^{ p}}{|s-t|^{\alpha p}}\,ds\,dt
& =
\iint_{\begin{subarray}{l} s\in I\\ t\in 1.1I
 \end{subarray}} \frac{|N(\gamma(s))- N(\gamma(t))|^p}{|s-t|^{\alpha p}}\,ds\,dt \\
 &\quad \!+  \iint_{\begin{subarray}{l} s\in I\\ t\in S^1(0,r_0)\setminus 1.1I
 \end{subarray}} \!\!\!\!\!\!\frac{|N(\gamma(s))- N(\gamma(t))|^p}{|s-t|^{\alpha p}}\,ds\,dt.
\end{align*}
To estimate the last integral we use the fact that $|s-t|\gtrsim \ell(I)$ in the domain
of integration. Indeed, for $s\in I_j$ we have
$$\int_{t\in S^1(0,r_0)\setminus 1.1I}\!\!\!
 \frac{|N(\gamma(s))- N(\gamma(t))|^p}{|s-t|^{\alpha p}}\,dt \leq \!
\int_{t\in S^1(0,r_0)\setminus 1.1I}\frac{2^p}{|s-t|^{\alpha p}}\,dt\lesssim 
\frac1{\ell(I)^{\alpha p-1}}.$$
From the last estimate and Lemma \ref{lempart} we obtain
\begin{align}\label{eqd890}
\iint_{\begin{subarray}{l} s\in I\\ t\in S^1(0,r_0)
 \end{subarray}} \!\!&\frac{|N(\gamma(s))- N(\gamma(t))|^p}{|s-t|^{\alpha p}}\,ds\,dt\lesssim \int_{
 \Omega\cap B(z_I,4\ell(I))} |\partial B(\chi_{\Omega})|^p\dist(\cdot,\partial
\Omega)^{p-\alpha p}\,dm\\
&\quad\quad+ \sum_{i=1}^2
\int_{|s-s_i|\leq 7a}\int_{t\in S^1(0,r_0)} \!\!\!\frac{|N(\gamma(s))-N(\gamma(t))|^p}{|s-t|^{
\alpha p}}\,dt\,ds
+c(\delta)\,\ell(I)^{2-\alpha p},\nonumber
\end{align}
where $s_1,s_2$ are the end points of $2I$ and $a=\delta^{1/2}\,\ell(I)$.

 Given $u\in S^1(0,r_0)$, we denote by $I_u$ the arc of $S^1(0,r_0)$ with length
$R/4$ whose mid point is $u$. 
Now we average the inequality \rf{eqd890} over all the intervals $I_u$, $u\in
S^1(0,r_0)$.
By Fubini, we have
\begin{multline*}
\int_{u\in S^1(0,r_0)}
\iint_{\begin{subarray}{l} s\in I_u\\ t\in S^1(0,r_0)
 \end{subarray}} \frac{|N(\gamma(s))- N(\gamma(t))|^p}{|s-t|^{\alpha p}}\,ds\,dt\,du
 \\ = 
 \frac{R}4
 \iint_{\begin{subarray}{l} s\in S^1(0,r_0)\\ t\in S^1(0,r_0)
 \end{subarray}} \frac{|N(\gamma(s))- N(\gamma(t))|^p}{|s-t|^{\alpha p}}\,
 ds\,dt =  \frac{R}4\,\|N\circ\gamma\|_{\dot{B}_{p,p}^{\alpha -1/p}}^p.
\end{multline*}

Now we turn to the right side of \rf{eqd890}. Concerning the first integral, we have
\begin{multline*}
\int_{u\in S^1(0,r_0)}\int_{\Omega\cap B(z_{I_u},4\ell(I))} \!\!|\partial B(\chi_{\Omega})|^p
\dist(\cdot,\partial\Omega)^{p-\alpha p}dm\,du
\\
= \int_\Omega\biggl(\int_{u:|\gamma(u)-w|\leq4\ell(I)}\!\!\!du\biggr)|\partial B(\chi_{\Omega})(w)|^p
\dist(w,\partial
\Omega)^{p-\alpha p}\,dm(w).
\end{multline*}
Taking into account that $\Omega$ is a Lipschitz domain, it follows that, for each $w\in\C$, 
$$\HH^1\{u:|\gamma(u)-w|\leq R\}\leq c\,R.$$
Thus,
\begin{multline*}
\int_{u\in S^1(0,r_0)}\int_{\Omega\cap B(z_{I_u},R)} |\partial B(\chi_{\Omega})|^p
\dist(\cdot,\partial
\Omega)^{p-\alpha p}\,dm\,du\\
\lesssim R\int_\Omega |\partial B(\chi_{\Omega})|^p\dist(\cdot,\partial
\Omega)^{p-\alpha p}\,
dm.
\end{multline*}

Now we consider the second term on the right side of \rf{eqd890}, for $i=1$, say:
\begin{align*}
\int_{u\in S^1(0,r_0)}\int_{|s-s_i|\leq 7a}&\int_{t\in S^1(0,r_0)} \!\!\!\frac{|N(\gamma(s))-N(\gamma(t))|^p}{|s-t|^{\alpha p}}\,dt\,ds\,du\\
& = 14a\int_{s\in S^1(0,r_0)}\int_{t\in S^1(0,r_0)} \!\!\!\frac{|N(\gamma(s))-N(\gamma(t))|^p}{|s-t|^{\alpha p}}\,dt\,ds\\
&\approx \delta^{1/2}R\,\|N\circ\gamma\|_{\dot{B}_{p,p}^{\alpha -1/p}(S^1(0,r_0))}^p.
\end{align*}
Clearly, the integral of the last term of \rf{eqd890} over $S^1(0,r_0)$ equals $c(\delta)R^{2-\alpha p}r_0$. Then finally we deduce that
\begin{align*}
\frac{R}4\,\|N\circ\gamma\|_{\dot{B}_{p,p}^{\alpha -1/p}(S^1(0,r_0))}^p&\lesssim R\,\|\partial B(\chi_{\Omega})\dist(\cdot,
\partial\Omega)^{1-\alpha }\|_{L^p(\Omega)}^p\\
&\quad
+ \delta^{1/2}R\,\|N\circ\gamma\|_{\dot{B}_{p,p}^{\alpha -1/p}(S^1(0,r_0))}^p + c(\delta)R^{2-\alpha p}r_0.
\end{align*}
Thus, if $\delta$ is taken small enough, then
$$\|N\circ\gamma\|_{\dot{B}_{p,p}^{\alpha -1/p}(S^1(0,r_0))}^p\lesssim \|\partial B(\chi_{\Omega})
\dist(\cdot,
\partial\Omega)^{1-\alpha }\|_{L^p(\Omega)}^p
+ c(\delta)R^{1-\alpha p}r_0,$$
and thus the theorem follows (in the case where $\Omega$ is simply connected).

\vspace{3mm}

If $\Omega$ is not simply connected, then $\partial\Omega$ has a finite number of components
(because it is a Lipschitz domain).
Arguing as above, one can show that \rf{eqd477} holds for the arc length parameterization $\gamma$
of each component, and then we are done.
\end{proof}


\section{The Beurling transform in $B_{p,p}^\alpha(\Omega)$}\label{sec6}

Recall that for a  Lipschitz or special Lipschitz domain and $f\in L^1_{loc}(\Omega)$, one
sets
$$\|f\|_{\dot{B}_{p,p}^\alpha(\Omega)}^p =
\iint \frac{|f(x)-f(y)|^p}{|x-y|^{\alpha p + 1}}\,dx\,dy.$$
In \cite{Cruz-Tolsa} it was shown that, for $0<\alpha<1$, the estimate \rf{eqct3} is also valid  with
$\|B(\chi_\Omega)\|_{\dot 
B_{p,p}^\alpha(\Omega)}$ replacing $\|B(\chi_\Omega)\|_{\dot W^{\alpha,p}(\Omega)}$.
That is, 
\begin{equation}\label{eqwwss}
\|B(\chi_\Omega)\|_{\dot B_{p,p}^\alpha(\Omega)}\leq c\,\|N\|_{\dot{B}_{p,p}^{\alpha-1/p}(\partial\Omega)}.
\end{equation}

One can also check that Theorems \ref{teodom} and \ref{teopri} also hold with
$\|B(\chi_\Omega)\|_{\dot 
B_{p,p}^\alpha(\Omega)}$ instead of $\|B(\chi_\Omega)\|_{\dot W^{\alpha,p}(\Omega)}$. So we have:

\begin{theorem}\label{teodom'}
Let
$0<\alpha<1$ and $1<p<\infty$ such that $\alpha p> 1$.
Let $\Omega\subset\C$ be either a $(\delta,R)$-Lipschitz domain or a 
$\delta$-Lipschitz domain, and assume that $\delta$ is small enough. If $\Omega$ is a Lipschitz domain, then
$$
\|N\|_{\dot B_{p,p}^{\alpha-1/p}(\partial\Omega)}\leq c\,\|B(\chi_\Omega)\|_{\dot 
B_{p,p}^\alpha(\Omega)}
+ \HH^1(\partial \Omega)^{-\alpha+2/p},
$$
and if it is a special Lipschitz domain,
$$\|N\|_{\dot B_{p,p}^{\alpha-1/p}(\partial\Omega)}\leq c\,\|B(\chi_\Omega)\|_{\dot 
B_{p,p}^\alpha(\Omega)}.$$
\end{theorem}

Observe that from \rf{eqwwss}, Theorem \ref{teodom'} and the analogous results involving
$\dot W^{\alpha,p}(\Omega)$ one deduces that, under the assumptions of Theorem \ref{teodom'},
$$B(\chi_\Omega) \in \dot B_{p,p}^\alpha(\Omega) \quad\Longleftrightarrow \quad
B(\chi_\Omega) \in \dot W^{\alpha,p}(\Omega).$$
Moreover, if $\Omega$ is a special Lipschitz domain,
$$\|B(\chi_\Omega)\|_{\dot 
B_{p,p}^\alpha(\Omega)} \approx \|B(\chi_\Omega)\|_{\dot W^{\alpha,p}(\Omega)}.$$

The proof of Theorem \ref{teodom'} follows by arguments very similar to the ones for 
Theorems \ref{teodom} and \ref{teopri}. One can check that the key Lemma \ref{lemkeyalfa}
also holds replacing $\|B(\chi_\Omega)\|_{\dot W^{\alpha,p}(\Omega)}$ by
$\|B(\chi_\Omega)\|_{\dot B_{p,p}^\alpha(\Omega)}$. Indeed, one writes
$$\|B\chi_\Omega\|_{\dot{B}_{p,p}^\alpha(\Omega)}^p \geq
\iint_{|x-y|\leq \ell(Q_x)} \frac{|B\chi_\Omega(x)-B\chi_\Omega(y)|^p}{|x-y|^{\alpha p + 1}}\,dx\,dy,$$
where $Q_x$ is the square from $\WW(\Omega)$ that contains $x$. Then, one uses the estimate 
\rf{eqff55} and then argues as in the proof of the lemma for the $\dot W^{\alpha,p}(\Omega)$ norm.
Applying this new version of Lemma \ref{lemkeyalfa} together with the Main Lemma \ref{mainlemma}
and Lemma \ref{mainlemma2}, Theorem \ref{teodom'} follows.


\vvv

\end{document}